\documentclass[12pt, reqno]{amsart}
\usepackage{mathrsfs}
\usepackage{amsfonts}
\usepackage[centertags]{amsmath}
\usepackage{amsthm}
\usepackage{amssymb}
\usepackage{graphicx}
\usepackage{caption}
\usepackage{resizegather}
\usepackage[colorlinks]{hyperref}
\usepackage{enumerate}
\usepackage[textwidth=16cm, hmarginratio=1:1]{geometry}
\usepackage{appendix}
\usepackage[all]{xy}
\usepackage{color}
\usepackage{float}
\usepackage{tikz}
\usepackage{mathabx}
\usetikzlibrary{matrix}
\usetikzlibrary{arrows.meta}
\usetikzlibrary{decorations.markings}
\usepackage{subcaption}
\usepackage{tikz-cd}
\usepackage{adjustbox}
\hypersetup{
	bookmarksnumbered,
	pdfstartview={FitH},
	breaklinks=true,
	linkcolor=blue,
	urlcolor=blue,
	citecolor=blue,
	bookmarksdepth=2
}

\newtheorem{theorem}{Theorem}[section]

\newtheorem{lemma}[theorem]{Lemma}
\newtheorem{proposition}[theorem]{Proposition}
\newtheorem{definition}{Definition}[section]
\newtheorem{remark}{Remark}[section]
\newtheorem{example}[theorem]{Example}

\newtheorem{problem}[theorem]{Problem}

\begin{document}
\title[Maximal green sequences for $\mathcal{Q}^N$ quivers]{Maximal green sequences for $\mathcal{Q}^N$ quivers}
\date{\today}

\author{Jingmin Guo}
\address{Jingmin Guo: School of Mathematics and Statistics, Lanzhou University, Lanzhou 730000, P. R. China.}
\email{guojm18@lzu.edu.cn}
	
\author{Bing Duan}
\address{Bing Duan: School of Mathematics and Statistics, Lanzhou University, Lanzhou 730000, P. R. China.}
\email{duanbing@lzu.edu.cn}
	
\author{Yanfeng Luo$^\dag$}
\address{Yanfeng Luo: School of Mathematics and Statistics, Lanzhou University, Lanzhou 730000, P. R. China.}
\email{luoyf@lzu.edu.cn}
\thanks{$\dag$ Corresponding author}

\date{}
	
\maketitle

\begin{abstract}
We introduce $\mathcal{Q}^N$ quivers and construct maximal green sequences for these quivers. We prove that any finite connected full subquiver of the quivers defined by Hernandez and Leclerc, arising in monoidal categorifications of cluster algebras, is a special case of $\mathcal{Q}^N$ quivers. Moreover, we prove that the trees of oriented cycles introduced by Garver and Musiker are special cases of $\mathcal{Q}^N$ quivers. This result resolves an open problem proposed by Garver and Musiker, providing a construction of maximal green sequences for quivers that are trees of oriented cycles. Furthermore, we prove that quivers that are mutation equivalent to an orientation of a type AD Dynkin diagram can also be recognized as special cases of $\mathcal{Q}^N$ quivers.
\end{abstract}

\hspace{0.15cm}
		
\noindent

		
\hspace{0.15cm}
		
\noindent


\tableofcontents
	
\setcounter{section}{0}
	
\allowdisplaybreaks
	
\section{Introduction}

Maximal green sequences are certain mutation sequences of framed cluster quivers. They were introduced by Keller in \cite{K11} to generate quantum dilogarithm identities and compute refined Donaldson-Thomas invariants.

The existence of maximal green sequences in a quiver has significant consequences. For example, a quiver with potential has a finite-dimensional Jacobian algebra if it admits such a sequence \cite{BDP14, K11}, which enables the construction of cluster categories, as introduced by Amiot \cite{A09}. In cluster algebra theory, maximal green sequences are associated with the relationship between cluster algebras and their upper cluster algebras \cite{M18}. They also support the construction of the theta basis in cluster algebras \cite{GHKK18} and good bases which contain the cluster monomials for a large class of upper cluster algebras \cite{Q22}.

In general, determining the existence of a maximal green sequence for a given quiver is a highly challenging task. In \cite{M16}, Muller proved that the existence of maximal green sequences is not invariant under quiver mutation and full subquivers inherit maximal green sequences. The study of the existence of maximal green sequences has been conducted in specific cases, such as quivers of finite type and acyclic quivers \cite{BDP14}, quivers arising from triangulations of surfaces \cite{B16, BM18, M17}, acyclic weighted quivers \cite{BHIT17}, and quivers arising from weighted projective lines \cite{FG23}. In \cite{BDP14, S14}, it was shown that certain quivers do not have a maximal green sequence. We refer the reader to \cite{K20} for recent progress. 

The purpose of this article is to construct explicit maximal green sequences for a class of quivers, including well-known examples such as quivers arising from monoidal categorifications of cluster algebras  \cite{HL16}, trees of oriented cycles \cite{GM17},  and quivers that are mutation equivalent to an orientation of a type AD Dynkin diagram \cite{BV08, V10}.

For $N\in\mathbb{Z}_{\geq1}$, we define the $\mathcal{Q}^N$ quivers $Q^N$, see Definition \ref{subsection: the definition of GHL-quiver}. We call a linearly oriented Dynkin quiver of type A a vertical chain. Then $Q^N$ is composed of $N$ vertical chains and arrows connecting these vertical chains. We introduce a partial order $>$ on the vertex set of $Q^N$, see Definition \ref{definition: partial order}. For each vertex of $Q^N$, we define an associated mutation sequence, see Definition \ref{The associated sequence}. By combining these mutation sequences with respect to the partial order defined in Definition \ref{definition: partial order}, we construct an explicit maximal green sequence for $Q^N$, see Theorem \ref{Main Theorem: maximal green sequences}. To verify this theorem, we first study quivers $Q^1$ and $Q^2$ in Section \ref{subsection: Two special classes of generalized quivers}. We prove that Theorem \ref{Main Theorem: maximal green sequences} holds for $Q^2$ in Lemma \ref{lemma of two vertical chains}. The proof of Theorem \ref{Main Theorem: maximal green sequences} heavily relies on the results of Lemma \ref{lemma of two vertical chains} as well as its proof.

To study the connections between finite-dimensional representations of quantum affine algebras and cluster algebras, Hernandez and Leclerc \cite{HL16} introduced quivers $G$ for every complex simple Lie algebra $\mathfrak{g}$ of finite type. These quivers play an important role in monoidal categorifications of cluster algebras. We prove that any finite connected full subquiver $\bar{G}$ of $G$ is a $\mathcal{Q}^N$ quiver, see Proposition \ref{HL quiver are GHL quiver}.  

Garver and Musiker introduced the concept of trees of $3$-cycles, and generalized it to trees of oriented cycles \cite{GM17}. They proposed an open problem, to find a construction of maximal green sequences for quivers that are trees of oriented cycles, see Problem \ref{Open problem1}. We prove that trees of oriented cycles are $\mathcal{Q}^N$ quivers, see Proposition \ref{prop:tree of oriented cycles}.  As a result,  Theorem \ref{Main Theorem: maximal green sequences} provides a solution for Problem \ref{Open problem1}. 
Moreover, we generalize the definition of trees of oriented cycles, and introduce of a class of irreducible quivers whose cycles are all oriented, we also prove them are $\mathcal{Q}^N$ quivers, see Proposition \ref{Prop:irreducible quivers contains oriented cycles}. 

We use $\mu^{A}$ and $\mu^{D}$ to denote the set of quivers that are mutation equivalent to orientations of type A and type D Dynkin diagrams, respectively. We prove that quivers in $\mu^A$ and $\mu^D$ are $\mathcal{Q}^N$ quivers, see Propositions \ref{prop:quiver in mu A} and \ref{theorem: a quiver of Type D is a generalized HL-quiver}. By applying Theorem \ref{Main Theorem: maximal green sequences}, we can get explicit maximal green sequences of quivers in $\mu^A$ and $\mu^D$.  Moreover, the maximal green sequences constructed for quivers in $\mu^A$ in this paper are different from those constructed in \cite{CDRSW16, GM17}, see Remark \ref{remark: the quiver in mu A}. Additionally, our construction method of the maximal green sequences for quivers in $\mu^D$ is different from the method of \cite{GMS18}.

The paper is organized as follows. In Section \ref{preface}, we review the notions of quiver mutations and maximal green sequences. In Section \ref{Section: Generalized quivers}, we introduce the $\mathcal{Q}^N$ quivers and construct explicit maximal green sequences for $\mathcal{Q}^N$ quivers. In Section \ref{Section: Applications}, we verify that any finite connected full subquiver of the quivers introduced by Hernandez and Leclerc can be regarded as a special case of $\mathcal{Q}^N$ quivers. Moreover, we prove that trees of oriented cycles, as well as quivers in $\mu^A$ and $\mu^D$, are special cases of $\mathcal{Q}^N$ quivers.

\section{Preliminary} \label{preface}
In this section, we recall the definitions of quivers, quiver mutations \cite{FZ02} and maximal green sequences \cite{BDP14, K11}.

\subsection{Quivers and quiver mutations}
A \textit{quiver} $Q$ is a directed graph. More precisely, $Q$ is defined as a $4$-tuple $(Q_0, Q_1, s, t)$, where $Q_0$ is a set of vertices, $Q_1$ is a set of arrows, and two functions $s, t: Q_1\rightarrow Q_0$ are defined so that $s(\alpha) \xrightarrow{\alpha} t(\alpha)$ for every $\alpha \in Q_1$.

A quiver $Q$ is a finite quiver if both $Q_0$ and $Q_1$ are finite sets. An arrow $\alpha$ is a loop if $s(\alpha)=t(\alpha)$; a $2$-cycle is a pair of distinct arrows $\beta$ and $\gamma$ such that $s(\beta)=t(\gamma)$ and $t(\beta)=s(\gamma)$. Throughout this paper, when we mention a quiver, it is implied to be a finite connected quiver without loops or $2$-cycles. A quiver is said to be irreducible if each arrow of this quiver is in some oriented cycle \cite{CL19, GM17}.

An \textit{ice quiver} is defined as a pair $(Q, F)$, where $Q$ is a quiver and $F \subset Q_0$ is a subset of vertices referred to as the \textit{frozen vertices} such that there are no arrows between them. The elements in $Q_0 \backslash F$ are known as mutable vertices. 
For an arrow $i \rightarrow j$ of $(Q, F)$, we say that $i$ points toward $j$, and $j$  points away $i$, and that $i$ and $j$ are neighbors.

Following \cite{FZ02}, let $(\mu_k(Q), F)$ be an ice quiver obtained from $(Q,F)$ by mutating at $k\in Q_0 \backslash F$, where $\mu_k(Q)$ is defined as follows:
\begin{enumerate}
\item for any two-arrow path $i\rightarrow k\rightarrow j$ in $Q$, adding a new arrow $i\rightarrow j$ unless $i, j \in F$;
\item reversing all arrows incident to $k$ in $Q$;
\item repeatedly removing oriented 2-cycles until unable to do so.
\end{enumerate}
Note that $\mu_k$ is an involution, that is, $\mu_k(\mu_k(Q))=Q$. In the sequel, we will avoid applying two consecutive mutations at the same vertex. Two quivers are \textit{mutation equivalent} if there exists a finite sequence of mutations sending one quiver to another quiver.

\subsection{Maximal green sequences}
Let $Q$ be a quiver with vertex set $Q_0=\{1,\dots,n\}$. Following \cite{BDP14, K11}, a framed (respectively, coframed) quiver $\widehat{Q}$ (respectively, $\widecheck{Q}$) of $Q$ is an ice quiver defined as follows:
\begin{align*}
& \widehat{Q}_0=\widecheck{Q}_0:=Q_0 \cup \{1',2',\dots,n'\}, \quad  F=\{1',2', \dots, n'\},\\
& \widehat{Q}_1=Q_1\sqcup\{i\rightarrow i'\mid 1\leq i\leq n\}  \quad (\text{respectively},~ \widecheck{Q}_1=Q_1\sqcup\{i\leftarrow i'\mid 1\leq i\leq n\}).
\end{align*}
Denote by $Q'$ a quiver obtained from $\widehat{Q}$ by applying a mutation sequence. We say that $i\in Q_0$ is \textit{green} in $Q'$ if
\[
\{j'\in F \mid \exists \, j' \rightarrow i\in Q'_1\}=\emptyset,
\]
and it is \textit{red} if
\[
\{j'\in F\mid \exists \, j'\leftarrow i\in Q'_1\}=\emptyset.
\]

The $\mathbf{g}$-vectors were introduced by Fomin and Zelevinsky to describe the degree vectors of cluster variables with principal coefficients \cite{FZ07}. It was proved in \cite[Theorem 1.7]{DWZ10} that $\mathbf{g}$-vectors are sign-coherent, which implies that every non-frozen vertex of $Q'$ is either red or green.

Let $\underline{\mu}=\mu_{\ell_k} \circ \mu_{\ell_{k-1}} \circ \dots \circ \mu_{\ell_1}$ be a mutation sequence of $\widehat{Q}$. Then the \textit{length} of $\underline{\mu}$ is defined to be $k$. In this paper, when we refer to mutating $\underline{\mu}$, it implies that the mutations are performed in a right-to-left order. If $\ell_j$ is a green vertex in $\mu_{\ell_{j-1}} \circ \mu_{\ell_{j-2}} \circ\dots \circ \mu_{\ell_1}(\widehat{Q})$ for $1\leq j\leq k$, then $\underline{\mu}$ is a \textit{green sequence} of $Q$. If $\underline{\mu}$ is a green sequence and the vertices $1,2, \dots, n$ are red in $\underline{\mu}(\widehat{Q})$, then $\underline{\mu}$ is a \textit{maximal green sequence} of $Q$. 

\section{Maximal green sequences for $\mathcal{Q}^N$ quivers}\label{Section: Generalized quivers}

In this section, we introduce  $\mathcal{Q}^N$ quivers and construct explicit maximal green sequences for them. 

\subsection{The definition of $\mathcal{Q}^N$ quivers} \label{subsection: the definition of GHL-quiver}
\begin{definition}\label{the definition of generalized HL-quiver}
For $N\in\mathbb{Z}_{\geq1}$, define a class of quivers $Q^N$ as finite connected quivers with the vertex sets $Q^N_0:=\{v^{i}_{\ell}\mid i=1,2,\dots,N, \,
\ell=1,2,\dots,k_i, \, k_{i}\in\mathbb{Z}_{\geq1}\}$, and the arrows of $Q^N$ are given as follows:
\begin{itemize}
\item [(1)]the vertical arrows
\[
v^{i}_{1}\leftarrow v^{i}_{2}\leftarrow \dots\leftarrow v^{i}_{k_i},\quad i=1,2,\dots,N;
\]
\item [(2)]the oblique arrows
\[
v^{i}_{a_1}\rightarrow v^{j}_{b_1} \rightarrow v^{i}_{a_2} \rightarrow v^{j}_{b_2} \rightarrow v^{i}_{a_3} \rightarrow v^{j}_{b_3}\rightarrow \cdots
\]
such that $i\neq j$, $1\leq i,j \leq N$, $1\leq a_1< a_2 <a_3< \dots \leq k_i$ and $1\leq b_1< b_2 <b_3< \dots \leq k_j$;
\item [(3)]if there is an arrow connecting $\{v^{i_r}_{\ell}\mid 1\leq\ell\leq k_{i_r}\}$ to $\{v^{i_{r+1}}_{\ell}\mid 1\leq\ell\leq k_{i_{r+1}}\}$ for $r=1,2,\dots, m-1$, then there is no arrow connecting $\{v^{i_1}_{\ell}\mid 1\leq\ell\leq k_{i_1}\}$ to $\{v^{i_{m}}_{\ell}\mid 1\leq\ell\leq k_{i_{m}}\}$, where $3\leq m \leq N$, $i_1\neq i_2\neq\dots\neq i_m$, $1\leq i_r \leq N$ for $r=1,2,\dots,m$.
\end{itemize}
For convenience, we call quivers $Q^N$, $N\in\mathbb{Z}_{\geq1}$, $\mathcal{Q}^N$ quivers.
\end{definition}

\begin{remark}
In $Q^N$, we call
\[
 \scalebox{0.6}{\xymatrix{
{\large\text{$v^{i}_1$}}\\
{\large\text{$v^{i}_2$}}\ar[u]\\
\vdots \ar[u]\\
{\large\text{$v^{i}_{k_i}$}} \ar[u]
}}
\]
a \textit{vertical chain}, the length of this vertical chain is defined to be $k_i$. By Definition \ref{the definition of generalized HL-quiver}, $Q^N$ is composed of $N$ vertical chains and oblique arrows connecting these vertical chains. Let 
\[
v^{i}_{1}\leftarrow v^{i}_{2} \leftarrow \dots \leftarrow v^{i}_{k_i}
\]
and 
\[
v^{j}_{1}\leftarrow v^{j}_{2}\leftarrow \dots \leftarrow v^{j}_{k_j}
\]
be two vertical chains of $Q^2$. It follows from Definition \ref{the definition of generalized HL-quiver} that the arrows connecting $\{v^{i}_{\ell}\mid 1\leq\ell\leq k_{i}\}$ to $\{v^{j}_{\ell}\mid 1\leq\ell\leq k_{j}\}$ exactly form a path. Then $Q^2$ must be one of quivers in Figure \ref{A Quiver of two vertical chains}.

\begin{figure}
\begin{minipage}[b]{0.48\linewidth}
\centering
\adjustbox{scale=0.4}{
\begin{tikzcd}
 {\huge\text{$v^{i}_1$}} \\
 {\huge\text{$v^{i}_2$}} &&{\huge\text{$v^{j}_{1}$}}  \\
 \vdots && {\huge\text{$v^{j}_2$}} \\
 {\huge\text{$v^{i}_{a_1}$}} && \vdots \\
 {\huge\text{$v^{i}_{a_1+1}$}}  && {\huge\text{$v^{j}_{b_1}$}} \\
 \vdots &&   {\huge\text{$v^{j}_{b_{1}+1}$}}   \\
 {\huge\text{$v^{i}_{a_2}$}} && \vdots \\
 \vdots && {\huge\text{$v^{j}_{b_{2}}$}}   \\
 {\huge\text{$v^{i}_{a_{n-1}}$}}  && \vdots \\
 {\huge\text{$v^{i}_{a_{n-1}+1}$}}  && {\huge\text{$v^{j}_{b_{n-1}}$}}  \\
 \vdots &&  {\huge\text{$v^{j}_{b_{n-1}+1}$}}  \\
  {\huge\text{$v^{i}_{a_{n}}$}}  && \vdots \\
 {\huge\text{$v^{i}_{a_{n}+1}$}}  && {\huge\text{$v^{j}_{b_{n}}$}} \\
 \vdots &&{\huge\text{$v^{j}_{b_{n}+1}$}}  \\
 {\huge\text{$v^{i}_{a_{n+1}}$}} && \vdots \\
 {} && {\huge\text{$v^{j}_{b_{n+1}}$}}
 \arrow[from=2-1, to=1-1]
 \arrow[from=15-1, to=14-1]
 \arrow[from=14-1, to=13-1]
 \arrow[from=13-1, to=12-1]
 \arrow[from=12-1, to=11-1]
 \arrow[from=11-1, to=10-1]
 \arrow[from=10-1, to=9-1]
 \arrow[from=9-1, to=8-1]
 \arrow[from=8-1, to=7-1]
 \arrow[from=7-1, to=6-1]
 \arrow[from=6-1, to=5-1]
 \arrow[from=5-1, to=4-1]
 \arrow[from=4-1, to=3-1]
 \arrow[from=3-1, to=2-1]
 \arrow[from=3-3, to=2-3]
 \arrow[from=4-3, to=3-3]
 \arrow[from=5-3, to=4-3]
 \arrow[from=6-3, to=5-3]
 \arrow[from=7-3, to=6-3]
 \arrow[from=8-3, to=7-3]
 \arrow[from=9-3, to=8-3]
 \arrow[from=10-3, to=9-3]
 \arrow[from=11-3, to=10-3]
 \arrow[from=12-3, to=11-3]
 \arrow[from=13-3, to=12-3]
 \arrow[from=4-1, to=5-3]
 \arrow[from=5-3, to=7-1]
 \arrow[from=7-1, to=8-3]
 \arrow[from=9-1, to=10-3]
 \arrow[from=10-3, to=12-1]
 \arrow[from=16-3, to=15-3]
 \arrow[from=15-3, to=14-3]
 \arrow[from=14-3, to=13-3]
 \arrow[from=12-1, to=13-3]
\end{tikzcd}
}\caption*{(a)}
\end{minipage}
\begin{minipage}[b]{0.48\linewidth}
\centering
\adjustbox{scale=0.4}{
\begin{tikzcd}
{\huge\text{$v^{i}_1$}} \\
{\huge\text{$v^{i}_2$}} &&{\huge\text{$v^{j}_{1}$}}  \\
 \vdots && {\huge\text{$v^{j}_2$}} \\
 {\huge\text{$v^{i}_{a_1}$}} && \vdots \\
 {\huge\text{$v^{i}_{a_1+1}$}}  && {\huge\text{$v^{j}_{b_1}$}} \\
 \vdots &&   {\huge\text{$v^{j}_{b_{1}+1}$}}   \\
 {\huge\text{$v^{i}_{a_2}$}} && \vdots \\
 \vdots && {\huge\text{$v^{j}_{b_{2}}$}}   \\
 {\huge\text{$v^{i}_{a_{n-1}}$}}  && \vdots \\
 {\huge\text{$v^{i}_{a_{n-1}+1}$}}  && {\huge\text{$v^{j}_{b_{n-1}}$}}  \\
 \vdots &&  {\huge\text{$v^{j}_{b_{n-1}+1}$}}  \\
  {\huge\text{$v^{i}_{a_{n}}$}}  && \vdots \\
 {\huge\text{$v^{i}_{a_{n}+1}$}}  && {\huge\text{$v^{j}_{b_{n}}$}} \\
 \vdots  \\
 {\huge\text{$v^{i}_{a_{n+1}}$}}
 \arrow[from=2-1, to=1-1]
 \arrow[from=15-1, to=14-1]
 \arrow[from=14-1, to=13-1]
 \arrow[from=13-1, to=12-1]
 \arrow[from=12-1, to=11-1]
 \arrow[from=11-1, to=10-1]
 \arrow[from=10-1, to=9-1]
 \arrow[from=9-1, to=8-1]
 \arrow[from=8-1, to=7-1]
 \arrow[from=7-1, to=6-1]
 \arrow[from=6-1, to=5-1]
 \arrow[from=5-1, to=4-1]
 \arrow[from=4-1, to=3-1]
 \arrow[from=3-1, to=2-1]
 \arrow[from=3-3, to=2-3]
 \arrow[from=4-3, to=3-3]
 \arrow[from=5-3, to=4-3]
 \arrow[from=6-3, to=5-3]
 \arrow[from=7-3, to=6-3]
 \arrow[from=8-3, to=7-3]
 \arrow[from=9-3, to=8-3]
 \arrow[from=10-3, to=9-3]
 \arrow[from=11-3, to=10-3]
 \arrow[from=12-3, to=11-3]
 \arrow[from=13-3, to=12-3]
 \arrow[from=4-1, to=5-3]
 \arrow[from=5-3, to=7-1]
 \arrow[from=7-1, to=8-3]
 \arrow[from=9-1, to=10-3]
 \arrow[from=10-3, to=12-1]
\end{tikzcd}
}\caption*{(b)}
\end{minipage}
\caption{ (a) Quiver  $Q^2$, where $k_i=a_{n+1}$, $k_j=b_{n+1}$. (b) Quiver  $Q^2$, where $k_i=a_{n+1}$, $k_j=b_{n}$.}
\label{A Quiver of two vertical chains}
\end{figure}
\end{remark}

\subsection{A partial order defined on the vertex sets of $\mathcal{Q}^N$ quivers}\label{section: partial order}

To construct a maximal green sequence for a $\mathcal{Q}^N$ quiver, we introduce a partial order on its vertex set.

\begin{definition}\label{definition: partial order}
For a quiver $Q^N$, we define a partial order $>$ on its vertex set as follows.
\begin{itemize}
\item[(1)] Let
\[
v^{i}_{1} \leftarrow v^{i}_{2}\leftarrow \dots \leftarrow v^{i}_{k_i}
\]
be a vertical chain of $Q^N$. Then
\[
v^i_1 > v^i_2> \dots > v^i_{k_i}.
\]
\item[(2)]Set $b_0=0$. If two vertical chains in $Q^N$ are connected by arrows as shown in Figure \ref{A Quiver of two vertical chains}(a), then 
\begin{gather}
\begin{align*}
&v^{i}_{a_{\ell}}>v^{j}_{b_{\ell-1}+1}> v^{j}_{b_{\ell-1}+2}>\dots>v^{j}_{b_{\ell}}\quad \text{for}\ \ell=1,2,\dots, n,\\
&v^{j}_{b_{\ell}}>v^{i}_{a_{\ell}+1}>v^{i}_{a_{\ell}+2}>\dots>v^{i}_{a_{\ell+1}} \quad \text{for}\ \ell=1,2,\dots, n-1,\\
&v^{j}_{b_n}>v^{i}_{a_{n}+1}>v^{i}_{a_{n}+2}>\dots>
v^{i}_{a_{n+1}}>v^{j}_{b_n+1}>v^{j}_{b_n+2}>\dots>v^{j}_{b_{n+1}}.
\end{align*}
\end{gather}
If two vertical chains in $Q^N$ are connected by arrows as shown in Figure \ref{A Quiver of two vertical chains}(b), then
\begin{gather}
\begin{align*}
&v^{i}_{a_{\ell}}>v^{j}_{b_{\ell-1}+1}> v^{j}_{b_{\ell-1}+2}>\dots>v^{j}_{b_{\ell}}\quad \text{for}\ \ell=1,2,\dots, n-1,\\
&v^{j}_{b_{\ell}}>v^{i}_{a_{\ell}+1}>v^{i}_{a_{\ell}+2}>\dots>v^{i}_{a_{\ell+1}}\quad \text{for}\ \ell=1,2,\dots, n-1,\\
&v^{i}_{a_n}>v^{j}_{b_{n-1}+1}>v^{j}_{b_{n-1}+2}>\dots>
v^{j}_{b_{n}}>v^{i}_{a_n+1}>v^{i}_{a_n+2}>\dots>v^{i}_{a_{n+1}}.
\end{align*}
\end{gather}
\end{itemize}
\end{definition}

We give an example to illustrate the definition of above partial order.

\begin{example}
In the following quiver, we have
\begin{align*}
&b_1>b_2>a_1>a_2>b_3>b_4>b_5>a_3>a_4>a_5>b_6>a_6,\\
&c_1>b_1>b_2>b_3>c_2>c_3>b_4>c_4>b_5>b_6>c_5.
\end{align*}
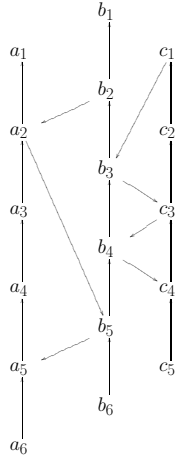
\[
 \scalebox{0.3}{\xymatrix{
&& {\text{\huge { $ b_1$ } }}\\
 {\text{\huge { $ a_1$ } }} &&&{\text{\huge { $ c_1$ } }}\ar[lddd] \\
&& {\text{\huge { $ b_2$ } }} \ar[uu]\ar[lld]\\
 {\text{\huge { $ a_2$ } }} \ar[uu] \ar[rrddddd]&& & {\text{\huge { $ c_2$ } }}\ar[uu]\\
&&  {\text{\huge { $ b_3$ } }} \ar[uu]\ar[rd]\\
 {\text{\huge { $ a_3$ } }} \ar[uu] && &{\text{\huge { $ c_3$ } }}\ar[uu]\ar[ld] \\
&& {\text{\huge { $ b_4$ } }} \ar[uu]\ar[rd]\\
 {\text{\huge { $ a_4$ } }} \ar[uu] &&& {\text{\huge { $ c_4$ } }}\ar[uu] \\
& &  {\text{\huge { $ b_5$ } }} \ar[uu]\ar[lld]\\
 {\text{\huge { $ a_5$ } }}  \ar[uu] &&& {\text{\huge { $ c_5$ } }}\ar[uu]  \\
& & {\text{\huge { $ b_6$ } }}  \ar[uu]\\
 {\text{\huge { $ a_6$ } }}  \ar[uu]& &\\
}}
\]
\end{example}

\subsection{A key lemma} \label{subsection: Two special classes of generalized quivers}
In this subsection, we study quivers $Q^1$ and $Q^2$.

Let $Q$ be a linearly oriented Dynkin quiver of type $A_n$, as shown below
\begin{align}\label{quiver Q^1}
1\leftarrow 2 \leftarrow \cdots  \leftarrow n.
\end{align}
As we know, $Q$ has a maximal green sequence given by  
\begin{align}\label{The mutation sequence of a vertical chain}
\mu_1\circ\mu_2\circ\mu_1\circ\dots\circ \mu_{n-1}\circ\dots\circ\mu_2\circ\mu_1\circ\mu_n\circ\dots\circ\mu_2\circ\mu_1
\end{align}

\begin{definition}\label{The associated sequence}
Let
\[
V:=v^i_1\leftarrow v^i_2 \leftarrow \dots \leftarrow v^i_{k_i}
\]
be a vertical chain of a quiver $Q^N$. For $1\leq \ell \leq k_i$, we define
\[
\underline{\mu}_{v^{i}_\ell}:=\mu_{v^i_{k_i-\ell+1}}\circ\mu_{v^i_{k_i-\ell}}\circ\dots\circ\mu_{v^i_1}.
\]
\end{definition}

\begin{remark}\label{The case of one vertical chain}
When we view $(\ref{quiver Q^1})$ as a quiver $Q^1$, the sequence $(\ref{The mutation sequence of a vertical chain})$ can be rewritten as follows
\[
\underline{\mu}:=\underline{\mu}_n\circ \underline{\mu}_{n-1}\circ \cdots \circ \underline{\mu}_2 \circ \underline{\mu}_1.
\]
Moreover, if $k$, $1\leq k \leq n$, is a mutable vertex of the sequence $\underline{\mu}$, after mutating all the vertices before $k$ in $\underline{\mu}$, then $k$ is a sink in the resulting quiver when we do not consider the frozen vertices,  and the arrows connecting the non-frozen vertices to $k$ are as follows (excluding the vertices that do not belong to the set $\{1,2,\dots,n\}$ or the arrows incident to it)
\[
 \scalebox{0.5}{\xymatrix{
 {\text{\large { $ k-1$ } }}\ar[d]\\
  {\text{\large { $ k $ } }}\\
 {\text{\large { $ k+1$ } }}\ar[u]
 }}
\]
Since $k$ is a green vertex, the arrows connecting the frozen vertices to $k$ point away $k$.
\end{remark}

In fact, in \cite{BMRYZ20}, the authors proved that the number of maximal green sequences of (\ref{quiver Q^1}) is at least $2^{n-1}$.

The following lemma is the main result of this subsection.
 
\begin{lemma}\label{lemma of two vertical chains}
Let $Q^2$ be the quiver (b) of Figure \ref{A Quiver of two vertical chains}. Then
\begin{gather}
\begin{align*}
\underline{\mu}=&\underline{\mu}_{v^{i}_{a_{n+1}}} \circ \underline{\mu}_{  v^{i}_{a_{n+1}-1} } \circ\cdots \circ  \underline{\mu}_{ v^{i}_{a_{n}+1} }\circ
\underline{\mu}_{  v^{j}_{b_{n}} } \circ\underline{\mu}_{  v^{j}_{b_{n}-1} } \circ  \cdots \circ \underline{\mu}_{  v^{j}_{b_{n-1}+1} } \circ
\underline{\mu}_{  v^{i}_{a_{n}} } \circ \underline{\mu}_{  v^{i}_{a_{n}-1} } \circ  \cdots  \\
&\circ\underline{\mu}_{  v^{i}_{a_{n-1}+1} }\circ
 \cdots \circ \underline{\mu}_{  v^{j}_{b_{2}} }\circ \underline{\mu}_{  v^{j}_{b_{2}-1} } \circ \cdots \circ \underline{\mu}_{  v^{j}_{b_{1}+1} } \circ
\underline{\mu}_{  v^{i}_{a_2} }\circ\underline{\mu}_{  v^{i}_{a_2-1} } \circ \cdots \circ \underline{\mu}_{  v^{i}_{a_1+1} }
 \circ\underline{\mu}_{  v^{j}_{b_{1}} }\circ\underline{\mu}_{  v^{j}_{b_{1}-1} } \\
&\circ \cdots\circ \underline{\mu}_{  v^{j}_{1} }\circ \underline{\mu}_{  v^{i}_{a_1} }\circ \underline{\mu}_{  v^{i}_{a_1-1} }\circ \cdots \circ  \underline{\mu}_{  v^{i}_{1} }
\end{align*}
\end{gather}
is a maximal green sequence of $Q^2$.
\end{lemma}

\begin{proof}
Let
\[
\underline{\mu}= \mu_{v_s}\circ \mu_{v_{s-1}} \circ \dots \circ \mu_{v_1}.
\]
Then $s=\frac{1}{2}\big(a_{n+1}(1+a_{n+1})+b_n(1+b_n)\big)$. In the following diagrams, for simplicity, we omit the frozen vertices. Set $a_0=b_0=0$. For $1\leq k \leq s$, we claim that:
\begin{itemize}
  \item [(a)] if
\[
\mu_{v_k}\circ \mu_{v_{k-1}} \circ \dots \circ \mu_{v_1}=
\mu_{ v^{i}_{\ell} }\circ\mu_{ v^{i}_{\ell-1} }\circ\cdots\circ\mu_{ v^{i}_{1} }\circ\underline{\mu}_{ v^{i}_{r} }\circ\cdots\circ\underline{\mu}_{ v^{i}_{a_{q}+1} }\underline{\mu}_{ v^{j}_{b_{q}} }\circ\cdots\underline{\mu}_{ v^{i}_1 },
\]
$0 \leq q \leq n$, $a_q\leq r< a_{q+1}$, and 
\begin{align*}
&1\leq\ell<a_{q+1}-r-1,\ \text{or} \\
&a_{q+t}-r\leq \ell< a_{q+t+1}-r-1\quad \text{for}\ 1\leq t\leq n-1-q,\ \text{or} \\
&a_{q+t}-r\leq \ell \leq a_{q+t+1}-r-1\quad \text{for}\ t=n-q,
\end{align*}
then $\mu_{v_k}\circ \mu_{v_{k-1}} \circ \dots \circ \mu_{v_1}(\widehat{Q^2})$ is the quiver shown in Figure \ref{claim figure at v p}(a), where
\[
\mathbf{v^{j}}=
\begin{cases}
    \textcolor{green}{v^{j}_{b_{q+t}-b_{q}}}&\text{if } \ell+1 \leq a_n-r, \\
    \textcolor{red}{v^{j}_{b_{n}-b_{q}+1}}&\text{if }  \ell+1  > a_n-r,
\end{cases}\quad  x=1-\delta_{\mathbf{v^{j}},\textcolor{red}{v^{j}_{b_{n}-b_{q}+1}}},
\quad y=1-\delta_{r,a_q},
\]
$\delta$ is the Kronecker delta;
  \item [(b)] if
\[
\mu_{v_k}\circ \mu_{v_{k-1}} \circ \dots \circ \mu_{v_1}=
\mu_{ v^{i}_{\ell} }\circ\mu_{ v^{i}_{\ell-1} }\circ\cdots\circ\mu_{ v^{i}_{1} }\circ\underline{\mu}_{ v^{i}_{r} }\circ\cdots\circ\underline{\mu}_{ v^{i}_{a_{q}+1} }\underline{\mu}_{ v^{j}_{b_{q}} }\circ\cdots\circ\underline{\mu}_{ v^{i}_1 },
\]
$0 \leq q \leq n$, $a_q\leq r< a_{q+1}$, $\ell= a_{q+t+1}-r-1$ for $0\leq t\leq n-1-q$, then $\mu_{v_k}\circ \mu_{v_{k-1}} \circ \dots \circ \mu_{v_1}(\widehat{Q^2})$ is the quiver shown in Figure \ref{claim figure at v p}(b),
where $y=1-\delta_{r,a_q}$;
  \item [(c)] if
\[
\mu_{v_k}\circ \mu_{v_{k-1}} \circ \dots \circ \mu_{v_1}=
\mu_{ v^{j}_{\ell} }\circ\mu_{ v^{j}_{\ell-1} }\circ\cdots\circ\mu_{ v^{j}_{1} }\circ\underline{\mu}_{ v^{j}_{r} }\circ\cdots\circ\underline{\mu}_{ v^{j}_{b_{q-1}+1} }\circ\underline{\mu}_{ v^{i}_{a_{q}} }\circ\cdots\circ\underline{\mu}_{ v^{i}_1 },
\]
$1 \leq q \leq n$, $b_{q-1} \leq r< b_{q}$, and
\begin{align*}
&1\leq\ell<b_{q}-r-1,\ \text{or} \\
&b_{q+t-1}-r \leq \ell < b_{q+t}-r-1\quad \text{for} \ 1\leq t\leq n-q-1,\ \text{or} \\
&b_{q+t-1}-r \leq \ell \leq b_{q+t}-r-1\quad \text{for} \ t= n-q,
 \end{align*}
then $\mu_{v_k}\circ \mu_{v_{k-1}} \circ \dots \circ \mu_{v_1}(\widehat{Q^2})$ is the quiver shown in Figure \ref{claim figure at v p}(c);
  \item [(d)] if
\[
\mu_{v_k}\circ \mu_{v_{k-1}} \circ \dots \circ \mu_{v_1}=
\mu_{ v^{j}_{\ell} }\circ\mu_{ v^{j}_{\ell-1} }\circ\cdots\circ\mu_{ v^{j}_{1} }\circ\underline{\mu}_{ v^{j}_{r} }\circ\cdots\circ\underline{\mu}_{ v^{j}_{b_{q-1}+1} }\circ\underline{\mu}_{ v^{i}_{a_{q}} }\circ\cdots\circ\underline{\mu}_{ v^{i}_1 },
\]
$1 \leq q \leq n$, $b_{q-1} \leq r< b_{q}$,  $\ell = b_{q+t}-r-1$ for $0\leq t\leq n-q-1$, then $\mu_{v_k}\circ \mu_{v_{k-1}} \circ \dots \circ \mu_{v_1}(\widehat{Q^2})$ is the quiver shown in Figure \ref{claim figure at v p}(d).
\end{itemize}

\begin{figure}
\resizebox{1\textwidth}{.8\height}{
\begin{minipage}[b]{0.25\textwidth}
\scalebox{0.4}{\xymatrix@R=10pt@C=0pt{
{\text{\huge \textcolor{green}{$v^{i}_1$} }}  &&   \\
\vdots \ar[u] & & {\text{\huge \textcolor{green}{$ v^{j}_1$} }} \\
{\text{\huge \textcolor{green}{$ v^{i}_{a_{q+1}-r-1}$} }}  \ar[u] \ar[rrd]&& \vdots  \ar[u]  \\
\vdots\ar[u] &&{\text{\huge \textcolor{green}{$v^{j}_{b_{q+1}-b_{q}} $} }}  \ar[lld] \ar[u]\\
{\text{\huge \textcolor{green}{$v^{i}_{a_{q+2}-r-1}$} }} \ar[u]\ar[rrd] & &\vdots\ar[u]\\
\vdots \ar[u] & \vdots & {\text{\huge \textcolor{green}{$ v^{j}_{b_{q+2}-b_{q}}$} }} \ar[u] \\
{\text{\huge \textcolor{green}{$v^{i}_{a_{q+t}-r-1}$} }}  \ar[rrd] \ar[u] & \vdots   &\vdots \ar[u]\\
\vdots\ar[u] &  &  {\text{\huge {$\mathbf{v^{j}}$} }}  \ar[u]\ar[lld]\ar[lldddd]^{\text{\huge $x$ }}   \\
{\text{\huge \textcolor{red}{$v^{i}_{\ell}$} }} \ar[d]  \ar[u] & &  \\
{\text{\huge \textcolor{green}{$v^{i}_{\ell+1}$} }} \ar[rruu] & &  \vdots\ar[uu] \\
\vdots \ar[u]&  &   \\
{\text{\huge \textcolor{green}{$v^{i}_{a_{q+t+1}-r}$} }} \ar[rrd] \ar[u]& &  \\
\vdots\ar[u]& &  {\text{\huge \textcolor{green}{$v^{j}_{b_{q+t+1}-b_{q}}$} }} \ar[uuu]\ar[lld] \\
{\text{\huge \textcolor{green}{$v^{i}_{a_{q+t+2}-r}$} }} \ar[rrd] \ar[u] &  &\vdots\ar[u]  \\
\vdots\ar[u]&\vdots & {\text{\huge \textcolor{green}{$v^{j}_{b_{q+t+2}-b_{q}}$} }}  \ar[u]   \\
{\text{\huge \textcolor{green}{$ v^{i}_{a_{n-1}-r}$} }} \ar[rrd]  \ar[u]  & \vdots  &\vdots\ar[u] \\
\vdots \ar[u]  & &  {\text{\huge \textcolor{green}{$ v^{j}_{b_{n-1}-b_{q}}$} }} \ar[u]\ar[lld] \\
{\text{\huge \textcolor{green}{$ v^{i}_{a_{n}-r}$} }}  \ar[rrdd]  \ar[u]  &  &\vdots\ar[u] \\
\vdots \ar[u] & &  {\text{\huge \textcolor{green}{$ v^{j}_{b_{n}-b_{q}}$} }} \ar[u] \\
{\text{\huge \textcolor{green}{$ v^{i}_{a_{n+1}-r}$} }} \ar[u] & &  {\text{\huge \textcolor{red}{$v^{j}_{b_{n}-b_{q}+1}$} }} \ar[u] \ar[dd] \ar[lld]_{{\text{\huge {$y $} }} }\\
{\text{\huge \textcolor{red}{$v^{i}_{a_{n+1}-r+1}$} }} \ar[u]\ar[d]  &  &  \\
\vdots\ar[d] & &  \vdots \ar[dd] \\
{\text{\huge \textcolor{red}{$ v^{i}_{a_{n+1}-a_{q}+1}$} }} \ar[d] \ar[rruuu]_{{\text{\huge {$y $} }}} &  & \\
\vdots\ar[d] & & {\text{\huge \textcolor{red}{$  v^{j}_{b_{n}-b_{q-1}+1}$} }}  \ar[llu]\ar[d]\\
{\text{\huge \textcolor{red}{$ v^{i}_{a_{n+1}-a_{q-1}+1} $} }} \ar[rru]\ar[d] &  &\vdots \ar[d] \\
\vdots\ar[d] & &  {\text{\huge \textcolor{red}{$v^{j}_{b_{n}-b_{q-2}+1}$} }} \ar[llu]\ar[d] \\
{\text{\huge \textcolor{red}{$ v^{i}_{a_{n+1}-a_{2}+1}$} }} \ar[d] &  &\vdots \ar[d] \\
\vdots\ar[d] & &  {\text{\huge \textcolor{red}{$v^{j}_{b_{n}-b_{1}+1}$} }} \ar[llu]\ar[d] \\
{\text{\huge \textcolor{red}{$v^{i}_{a_{n+1}-a_{1}+1}$} }} \ar[rru]\ar[d] &  &\vdots \ar[d] \\
\vdots\ar[d] & &  {\text{\huge \textcolor{red}{$ v^{j}_{b_{n}}$} }}  \\
{\text{\huge \textcolor{red}{$ v^{i}_{a_{n+1}} $} }}  & &  \\
}}
\caption*{(a)}
\end{minipage}\begin{minipage}[b]{0.25\textwidth}
\scalebox{0.4}{\xymatrix@R=10pt@C=0pt{
{\text{\huge \textcolor{green}{$v^{i}_1$} }}  &&   \\
\vdots \ar[u] & & {\text{\huge \textcolor{green}{$ v^{j}_1$} }} \\
{\text{\huge \textcolor{green}{$ v^{i}_{a_{q+1}-r-1}$} }}  \ar[u] \ar[rrd]&& \vdots  \ar[u]  \\
\vdots\ar[u] &&{\text{\huge \textcolor{green}{$v^{j}_{b_{q+1}-b_{q}} $} }}  \ar[lld] \ar[u]\\
{\text{\huge \textcolor{green}{$v^{i}_{a_{q+2}-r-1}$} }} \ar[u]\ar[rrd] & &\vdots\ar[u]\\
\vdots \ar[u] & \vdots & {\text{\huge \textcolor{green}{$ v^{j}_{b_{q+2}-b_{q}}$} }} \ar[u] \\
{\text{\huge \textcolor{green}{$v^{i}_{a_{q+t}-r-1}$} }}  \ar[rrd] \ar[u] & \vdots   &\vdots \ar[u]\\
&  &  {\text{\huge \textcolor{green}{$v^{j}_{b_{q+t}-b_{q}}$} }}  \ar[u]\ar[lldd]   \\
\vdots\ar[uu] & &  \\
{\text{\huge \textcolor{red}{$v^{i}_{\ell}$} }} \ar[dd]\ar[u] & &  \vdots\ar[uu] \\
&  &   \\
{\text{\huge \textcolor{green}{$v^{i}_{a_{q+t+1}-r}$} }} \ar[rrd] & &  \\
\vdots\ar[u]& &  {\text{\huge \textcolor{green}{$v^{j}_{b_{q+t+1}-b_{q}}$} }} \ar[uuu]\ar[lld] \\
{\text{\huge \textcolor{green}{$v^{i}_{a_{q+t+2}-r}$} }} \ar[rrd] \ar[u] &  &\vdots\ar[u]  \\
\vdots\ar[u]&\vdots & {\text{\huge \textcolor{green}{$v^{j}_{b_{q+t+2}-b_{q}}$} }}  \ar[u]   \\
{\text{\huge \textcolor{green}{$ v^{i}_{a_{n-1}-r}$} }} \ar[rrd]  \ar[u]  & \vdots  &\vdots\ar[u] \\
\vdots \ar[u]  & &  {\text{\huge \textcolor{green}{$ v^{j}_{b_{n-1}-b_{q}}$} }} \ar[u]\ar[lld] \\
{\text{\huge \textcolor{green}{$ v^{i}_{a_{n}-r}$} }}  \ar[rrdd]  \ar[u]  &  &\vdots\ar[u] \\
\vdots \ar[u] & &  {\text{\huge \textcolor{green}{$ v^{j}_{b_{n}-b_{q}}$} }} \ar[u] \\
{\text{\huge \textcolor{green}{$ v^{i}_{a_{n+1}-r}$} }} \ar[u] & &  {\text{\huge \textcolor{red}{$v^{j}_{b_{n}-b_{q}+1}$} }} \ar[u] \ar[dd] \ar[lld]_{{\text{\huge {$ y $} }}}\\
{\text{\huge \textcolor{red}{$v^{i}_{a_{n+1}-r+1} $} }}\ar[u]\ar[d]  &  &  \\
\vdots\ar[d] & &  \vdots \ar[dd] \\
{\text{\huge \textcolor{red}{$ v^{i}_{a_{n+1}-a_{q}+1}$} }} \ar[d] \ar[rruuu]_{{\text{\huge {$y$} }}} &  & \\
\vdots\ar[d] & & {\text{\huge \textcolor{red}{$  v^{j}_{b_{n}-b_{q-1}+1}$} }}  \ar[llu]\ar[d]\\
{\text{\huge \textcolor{red}{$ v^{i}_{a_{n+1}-a_{q-1}+1} $} }} \ar[rru]\ar[d] &  &\vdots \ar[d] \\
\vdots\ar[d] & &  {\text{\huge \textcolor{red}{$v^{j}_{b_{n}-b_{q-2}+1}$} }} \ar[llu]\ar[d] \\
{\text{\huge \textcolor{red}{$ v^{i}_{a_{n+1}-a_{2}+1}$} }} \ar[d] &  &\vdots \ar[d] \\
\vdots\ar[d] & &  {\text{\huge \textcolor{red}{$v^{j}_{b_{n}-b_{1}+1}$} }} \ar[llu]\ar[d] \\
{\text{\huge \textcolor{red}{$v^{i}_{a_{n+1}-a_{1}+1}$} }} \ar[rru]\ar[d] &  &\vdots \ar[d] \\
\vdots\ar[d] & &  {\text{\huge \textcolor{red}{$ v^{j}_{b_{n}}$} }}  \\
{\text{\huge \textcolor{red}{$ v^{i}_{a_{n+1}} $} }}  &  &  \\
}}
\vspace{0.15cm}
\caption*{(b)}
\end{minipage}\begin{minipage}[b]{0.25\textwidth}
\scalebox{0.4}{\xymatrix@R=10pt@C=0pt{
{\text{\huge \textcolor{green}{$v^{i}_1$} }}  &&   \\
\vdots \ar[u] & & {\text{\huge \textcolor{green}{$ v^{j}_1$} }} \\
{\text{\huge \textcolor{green}{$v^{i}_{a_{q+1}-a_{q}}$} }}  \ar[u] \ar[rrddd]&& \vdots  \ar[u]  \\
\vdots\ar[u] &&{\text{\huge \textcolor{green}{$ v^{j}_{b_{q}-r-1}$} }}   \ar[llu] \ar[u]\\
{\text{\huge \textcolor{green}{$v^{i}_{a_{q+2}-a_{q}}$} }}\ar[u] & &\vdots\ar[u]\\
\vdots \ar[u] & \vdots & {\text{\huge \textcolor{green}{$ v^{j}_{b_{q+1}-r-1}$} }}\ar[llu]\ar[u] \\
{\text{\huge \textcolor{green}{$v^{i}_{a_{q+t}-a_{q}}$} }}\ar[rrddd]\ar[rrdddddd] \ar[u] & \vdots   &\vdots \ar[u]\\
&  &  {\text{\huge \textcolor{green}{$v^{j}_{b_{q+t-1}-r-1}$} }} \ar[u]\ar[llu]   \\
\vdots\ar[uu] & & \vdots\ar[u] \\
& & {\text{\huge \textcolor{red}{$ v^{j}_{\ell}$} }}\ar[d]\ar[u] \\
&  &   {\text{\huge \textcolor{green}{$v^{j}_{\ell+1}$} }}\ar[lluuuu] \\
{\text{\huge \textcolor{green}{$v^{i}_{a_{q+t+1}-a_{q}} $} }} \ar[uuu]\ar[rrddd]  & &  \vdots\ar[u]\\
\vdots\ar[u]& &  {\text{\huge \textcolor{green}{$v^{j}_{b_{q+t}-r}$} }}\ar[u]\ar[llu] \\
{\text{\huge \textcolor{green}{$v^{i}_{a_{q+t+2}-a_{q}}$} }}   \ar[u] &  &\vdots\ar[u]  \\
\vdots\ar[u]&\vdots & {\text{\huge \textcolor{green}{$v^{j}_{b_{q+t+1}-r}$} }} \ar[u] \ar[llu]  \\
{\text{\huge \textcolor{green}{$v^{i}_{a_{n}-a_{q}}$} }} \ar[rrddddd] \ar[u]  & \vdots  &\vdots\ar[u] \\
\vdots \ar[u]  & &  {\text{\huge \textcolor{green}{$ v^{j}_{b_{n-1}-r}$} }}\ar[u]\ar[llu] \\
{\text{\huge \textcolor{green}{$v^{i}_{a_{n+1}-a_{q}}$} }} \ar[u]  &  &\vdots\ar[u] \\
{\text{\huge \textcolor{red}{$ v^{i}_{a_{n+1}-a_{q}+1}$} }} \ar[u]\ar[dd]  & & {\text{\huge \textcolor{green}{$v^{j}_{b_{n}-r}  $} }}\ar[u] \\
  & &  \\
\vdots\ar[d] & &  {\text{\huge \textcolor{red}{$ v^{j}_{b_{n}-r+1}$} }} \ar[uu]\ar[d]\\
{\text{\huge \textcolor{red}{$ v^{i}_{a_{n+1}-a_{q-1}+1} $} }} \ar[dd] \ar[rrdd] &  & \vdots\ar[dd] \\
 & &   \\
\vdots\ar[d] & & {\text{\huge \textcolor{red}{$  v^{j}_{b_{n}-b_{q-1}+1}$} }}  \ar[lluuuuu]\ar[d]\\
{\text{\huge \textcolor{red}{$v^{i}_{a_{n+1}-a_{q-2}+1}$} }}\ar[rrdd]\ar[dd] &  &\vdots \ar[dd] \\
 & &   \\
\vdots\ar[d] & &  {\text{\huge \textcolor{red}{$v^{j}_{b_{n}-b_{q-2}+1}$} }}\ar[lluuuuu]\ar[d] \\
{\text{\huge \textcolor{red}{$v^{i}_{a_{n+1}-a_{1}+1}$} }} \ar[rrdd]\ar[dd] &  &\vdots \ar[dd] \\
& &  \\
\vdots\ar[dd] & &  {\text{\huge \textcolor{red}{$v^{j}_{b_{n}-b_{1}+1}$} }} \ar[d]  \\
&  &\vdots \ar[d] \\
{\text{\huge \textcolor{red}{$v^{i}_{a_{n+1}}$} }}     & &  {\text{\huge \textcolor{red}{$ v^{j}_{b_{n}}$} }}  \\
}}
\vspace{0.2cm}
\caption*{(c)}
\end{minipage}\begin{minipage}[b]{0.25\textwidth}
\scalebox{0.4}{\xymatrix@R=10pt@C=0pt{
{\text{\huge \textcolor{green}{$v^{i}_1$} }}  &&   \\
\vdots \ar[u] & & {\text{\huge \textcolor{green}{$ v^{j}_1$} }} \\
{\text{\huge \textcolor{green}{$v^{i}_{a_{q+1}-a_{q}}$} }}  \ar[u] \ar[rrddd]&& \vdots  \ar[u]  \\
\vdots\ar[u] &&{\text{\huge \textcolor{green}{$ v^{j}_{b_{q}-r-1}$} }}   \ar[llu] \ar[u]\\
{\text{\huge \textcolor{green}{$v^{i}_{a_{q+2}-a_{q}}$} }}\ar[u] & &\vdots\ar[u]\\
\vdots \ar[u] & \vdots & {\text{\huge \textcolor{green}{$ v^{j}_{b_{q+1}-r-1}$} }}\ar[llu]\ar[u] \\
{\text{\huge \textcolor{green}{$v^{i}_{a_{q+t}-a_{q}}$} }}\ar[rrdddd] \ar[u] & \vdots   &\vdots \ar[u]\\
&  &  {\text{\huge \textcolor{green}{$v^{j}_{b_{q+t-1}-r-1}$} }} \ar[u]\ar[llu]   \\
\vdots\ar[uu] & & \vdots\ar[u] \\
& & \\
&  &  {\text{\huge \textcolor{red}{$ v^{j}_{\ell}$} }}\ar[dd]\ar[uu] \\
{\text{\huge \textcolor{green}{$v^{i}_{a_{q+t+1}-a_{q}} $} }} \ar[uuu]\ar[rrddd]  & &   \\
\vdots\ar[u]& &  {\text{\huge \textcolor{green}{$v^{j}_{b_{q+t}-r}$} }} \ar[llu] \\
{\text{\huge \textcolor{green}{$v^{i}_{a_{q+t+2}-a_{q}}$} }}   \ar[u] &  &\vdots\ar[u]  \\
\vdots\ar[u]&\vdots & {\text{\huge \textcolor{green}{$v^{j}_{b_{q+t+1}-r}$} }} \ar[u] \ar[llu]  \\
{\text{\huge \textcolor{green}{$v^{i}_{a_{n}-a_{q}}$} }} \ar[rrddddd] \ar[u]  & \vdots  &\vdots\ar[u] \\
\vdots \ar[u]  & &  {\text{\huge \textcolor{green}{$ v^{j}_{b_{n-1}-r}$} }}\ar[u]\ar[llu] \\
{\text{\huge \textcolor{green}{$v^{i}_{a_{n+1}-a_{q}}$} }} \ar[u]  &  &\vdots\ar[u] \\
{\text{\huge \textcolor{red}{$ v^{i}_{a_{n+1}-a_{q}+1}$} }} \ar[u]\ar[dd] & & {\text{\huge \textcolor{green}{$v^{j}_{b_{n}-r}  $} }}\ar[u] \\
 & &   \\
\vdots\ar[d] & &  {\text{\huge \textcolor{red}{$ v^{j}_{b_{n}-r+1}$} }}\ar[uu] \ar[d]\\
{\text{\huge \textcolor{red}{$ v^{i}_{a_{n+1}-a_{q-1}+1} $} }} \ar[dd] \ar[rrdd] &  & \vdots\ar[dd] \\
 & &  \\
\vdots\ar[d] & & {\text{\huge \textcolor{red}{$  v^{j}_{b_{n}-b_{q-1}+1}$} }}  \ar[lluuuuu]\ar[d]\\
{\text{\huge \textcolor{red}{$v^{i}_{a_{n+1}-a_{q-2}+1}$} }}\ar[rrdd]\ar[dd] &  &\vdots \ar[dd] \\
 & &   \\
\vdots\ar[d] & &  {\text{\huge \textcolor{red}{$v^{j}_{b_{n}-b_{q-2}+1}$} }}\ar[lluuuuu]\ar[d] \\
{\text{\huge \textcolor{red}{$v^{i}_{a_{n+1}-a_{1}+1}$} }} \ar[rrdd]\ar[dd] &  &\vdots \ar[dd] \\
  & &  \\
\vdots\ar[dd] & &  {\text{\huge \textcolor{red}{$v^{j}_{b_{n}-b_{1}+1}$} }} \ar[d]  \\
&  &\vdots \ar[d] \\
{\text{\huge \textcolor{red}{$v^{i}_{a_{n+1}}$} }}     & &  {\text{\huge \textcolor{red}{$ v^{j}_{b_{n}}$} }}  \\
}}
\vspace{0.45cm}
\caption*{(d)}
\end{minipage}}
\caption{The quiver $\mu_{v_k}\circ \mu_{v_{k-1}} \circ \dots \circ \mu_{v_1}(\widehat{Q^2})$. }\label{claim figure at v p}
\end{figure}

By our claim, if we do not consider the frozen vertices, then the arrows that are incident to the vertex $v_k$ in the quiver $\mu_{v_{k-1}}  \circ \mu_{v_{k-2}} \circ \cdots \circ \mu_{v_1} (\widehat{Q^2})$ form a full subquiver of the quiver
\[
 \scalebox{0.5}{\xymatrix{
 {\text{\large { $\alpha_1$ } }} \ar[d]&\\
 {\text{\large { $v_{k}$ } }}  \ar[r]& {\text{\large { $\beta$ } }} \\
 {\text{\large { $\alpha_2$ } }} \ar[u]&
}}
\]
such that $\alpha_2 \rightarrow v_{k} \rightarrow \alpha_1$ is a full subquiver of one vertical chain of $Q^2$, and $\beta$ is a vertex of another vertical chain of $Q^2$.
Then we have:
\begin{itemize}
\item[(1)] the arrows connecting the vertices that belong to the same vertical chain as $v_{k}$ to $v_{k}$ point toward $v_{k}$;
\item[(2)] the arrow connecting a vertex of other vertical chain to $v_{k}$ points away $v_{k }$.
\end{itemize}
Moreover, since each mutable vertex $v_k$ is a green vertex, we have:
\begin{itemize}
\item[(3)] the arrows connecting the frozen vertices to $v_k$ point away $v_k$.
\end{itemize}
Define
\begin{gather}
\begin{align*}
Q^1:= v^{i}_1 \leftarrow  v^{i}_2 \leftarrow \dots \leftarrow  v^{i}_{a_{n+1}}\quad(\text{respectively},~Q^1:= v^{j}_1 \leftarrow  v^{j}_2 \leftarrow \dots \leftarrow  v^{j}_{b_n}).
\end{align*}
\end{gather}
It follows from (1), (2) and (3) that the full subquiver of
\[
\underline{\mu}(\widehat{Q^2})
\]
with the vertex set $\widehat{Q^1_0}=\{v^{i}_{\ell}, {v^{i}_{\ell}}'\mid \ell=1,2,\dots,a_{n+1}\}$ (respectively, $\widehat{Q^1_0}=\{v^{j}_{\ell},{v^{j}_{\ell}}'\mid \ell=1,2,\dots,b_{n}\}$) is the same as  quiver
\[
\underline{\mu}\vert_{Q^1_0 } (\widehat{Q^1}),
\]
where $\underline{\mu}\vert_{ Q^1_0}$ is a mutation sequence obtained from $\underline{\mu}$ by deleting the vertices that do not belong to $\{v^{i}_{\ell}\mid \ell=1,2,\dots,a_{n+1}\}$ (respectively, $\{v^{j}_{\ell}\mid \ell=1,2,\dots,b_{n}\}$). We know that
\begin{align*}
\underline{\mu}\vert_{ Q^1_0 }&=\underline{\mu}_{v^{i}_{a_{n+1}}}\circ \underline{\mu}_{v^{i}_{a_{n+1}-1}}\circ\dots\circ\underline{\mu}_{v^{i}_{1}}
(\text{respectively},\, \underline{\mu}\vert_{ Q^1_0 }=\underline{\mu}_{v^{j}_{b_{n}}}\circ \underline{\mu}_{v^{j}_{b_{n}-1}}\circ\dots\circ\underline{\mu}_{v^{j}_{1}})
\end{align*}
is a maximal green sequence of $Q^1$. Moreover, there is no arrow connecting the vertices in $\{v^{j}_{\ell}\mid \ell=1,2,\dots,b_{n}\}$ (respectively, $\{v^{i}_{\ell} \mid \ell=1,2,\dots,a_{n+1}\}$) to the vertices in $\{{v^{i}_{\ell}}'\mid \ell=1,2,\dots,a_{n+1}\}$ (respectively, $\{{v^{j}_{\ell}}'\mid \ell=1,2,\dots,b_{n}\}$). 
Therefore, our result holds if the claim is true.

We prove our claim by induction on $k$. 

{\bf Case 1.} When $k=1$.
If $a_1>1$, since $v^{i}_1$ is a sink in $Q^2$, we can directly deduce that the quiver $\mu_{v^{i}_1}(\widehat{Q^2})$ is the desired quiver.
If $a_1=1$, after mutating the vertex $v^{i}_1$, the quiver $\mu_{v^{i}_1}(\widehat{Q^2})$ is obtained, as required, see Figure \ref{case 1 of a1=1}.

\begin{figure}
\hskip 0.5 cm
\begin{minipage}[t]{0.42\textwidth}
 \scalebox{0.4}{\xymatrix{
{\text{\huge \textcolor{green}{$ v^{i}_{a_{1}}$} }}\ar[rrdd]&& {\text{\huge \textcolor{green}{$ v^{j}_1$} }} && {\text{\huge \textcolor{red}{$v^{i}_{a_{1}}$} }}\ar[d]&& {\text{\huge \textcolor{green}{$ v^{j}_1$} }} \\
{\text{\huge \textcolor{green}{$ v^{i}_{a_{1}+1} $} }}\ar[u]  && \vdots  \ar[u] && {\text{\huge \textcolor{green}{$v^{i}_{a_{1}+1}$} }}\ar[rrd]  && \vdots  \ar[u] \\
\vdots\ar[u] &&{\text{\huge \textcolor{green}{$v^{j}_{b_{1}} $} }} \ar[lld] \ar[u]&\xrightarrow{\substack{  {\text{\huge  {$\mu_{ v^{i}_{a_1} }$} }}  }}& \vdots\ar[u] &&{\text{\huge \textcolor{green}{$v^{j}_{b_{1}} $} }}\ar[lluu] \ar[lld] \ar[u]\\
{\text{\huge \textcolor{green}{$v^{i}_{a_{2}}  $} }}  \ar[u]\ar[rrd]  &&\vdots\ar[u] &&   {\text{\huge \textcolor{green}{$v^{i}_{a_{2}}  $} }}\ar[u]\ar[rrd] &&\vdots\ar[u] \\
\vdots\ar[u]  && {\text{\huge \textcolor{green}{$v^{j}_{b_{2}} $} }}\ar[u] &&   \vdots\ar[u]&&{\text{\huge \textcolor{green}{$v^{j}_{b_{2}} $} }}\ar[u] \\
 &&\vdots\ar[u] &\text{\huge(a)}&   &&\vdots\ar[u] \\
}}
\end{minipage}
\hskip 0.5 cm
\begin{minipage}[t]{0.42\textwidth}
 \scalebox{0.4}{\xymatrix{
{\text{\huge \textcolor{green}{$ v^{i}_{a_{1}}$} }} \ar[rrdd]&& {\text{\huge \textcolor{green}{$ v^{j}_1$} }} && {\text{\huge \textcolor{red}{$v^{i}_{a_{1}}$} }}\ar[ddd]&& {\text{\huge \textcolor{green}{$ v^{j}_1$} }} \\
&& \vdots  \ar[u] && && \vdots  \ar[u] \\
&&{\text{\huge \textcolor{green}{$v^{j}_{b_{1}} $} }} \ar[lld] \ar[u]& \xrightarrow{\substack{  {\text{\huge  {$\mu_{ v^{i}_{a_1} }$} }}  }} &  &&{\text{\huge \textcolor{green}{$v^{j}_{b_{1}} $} }}\ar[lluu]  \ar[u]\\
{\text{\huge \textcolor{green}{$v^{i}_{a_{2}}  $} }}\ar[uuu]\ar[rrd]  &&\vdots\ar[u] &&   {\text{\huge \textcolor{green}{$v^{i}_{a_{2}}  $} }} \ar[rrd] &&\vdots\ar[u] \\
\vdots\ar[u]  &&{\text{\huge \textcolor{green}{$v^{j}_{b_{2}} $} }}\ar[u] &&   \vdots\ar[u]&&{\text{\huge \textcolor{green}{$v^{j}_{b_{2}} $} }}\ar[u] \\
 &&\vdots\ar[u] &\text{\huge(b)}&   &&\vdots\ar[u] \\
}}
\end{minipage}
\caption{ (a) is the mutation at  $v^{i}_{a_1}$ in $\widehat{Q^2}$, where $a_{1}=1$, $a_1<a_2-1$;  (b) is the mutation at  $v^{i}_{a_1}$ in $\widehat{Q^2}$, where $a_{1}=1$, $a_1=a_2-1$.}\label{case 1 of a1=1}
\end{figure}
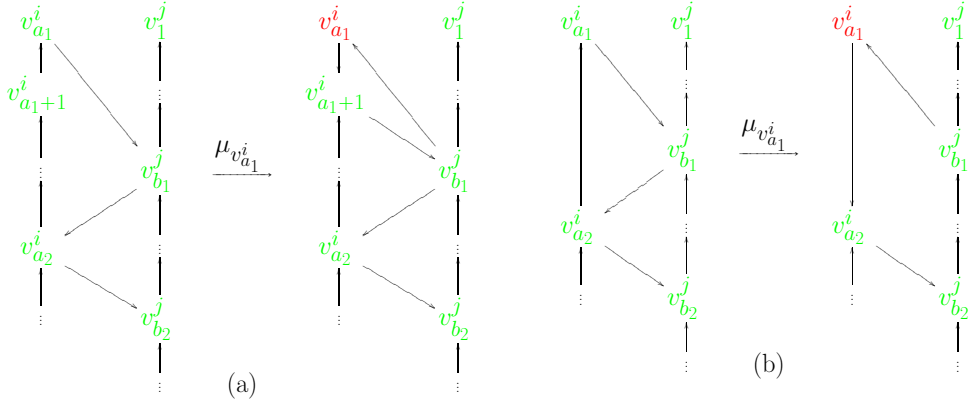

{\bf Case 2.} Suppose that for $1\leq k\leq s-1$, the quiver $\mu_{v_k}\circ \mu_{v_{k-1}} \circ \dots \circ \mu_{v_1}(\widehat{Q^2})$ satisfies our claim. By induction, we need to prove that our result holds for $\mu_{v_{k+1}}\circ \mu_{v_{k}} \circ \dots \circ \mu_{v_1}(\widehat{Q^2})$. By assumption, the quiver $\mu_{v_k}\circ \mu_{v_{k-1}} \circ \dots \circ \mu_{v_1}(\widehat{Q^2})$ is one of the quivers shown in Figure \ref{claim figure at v p}. When $\mu_{v_k}\circ \mu_{v_{k-1}} \circ \dots \circ \mu_{v_1}(\widehat{Q^2})$ is the quiver shown in Figure \ref{claim figure at v p}(a), after mutating the vertex $ v_{k+1}$, the quiver $\mu_{v_{k+1}}\circ \mu_{v_{k}} \circ \dots \circ \mu_{v_1}(\widehat{Q^2})$ is obtained, as required, see Figure \ref{claim figure at v k+1 (a)}. When $\mu_{v_k}\circ \mu_{v_{k-1}} \circ \dots \circ \mu_{v_1}(\widehat{Q^2})$ is the quiver (b), (c), or (d) of Figure \ref{claim figure at v p}, after mutating the vertex $v_{k+1}$, the obtained quiver $\mu_{v_{k+1}}\circ \mu_{v_{k}} \circ \dots \circ \mu_{v_1}(\widehat{Q^2})$ is still desired. The proof is complete.
\end{proof}
\begin{figure}
\resizebox{1\textwidth}{.8\height}{
\begin{minipage}[b]{0.5\textwidth}
\scalebox{0.35}{\xymatrix@R=10pt@C=0pt{
{\text{\huge \textcolor{green}{$v^{i}_1$} }}  &&   &&{\text{\huge \textcolor{green}{$v^{i}_1$} }}  &&\\
 \vdots \ar[u] & & {\text{\huge \textcolor{green}{$ v^{j}_1$} }} && \vdots \ar[u] & & {\text{\huge \textcolor{green}{$ v^{j}_1$} }} \\
{\text{\huge \textcolor{green}{$ v^{i}_{a_{q+1}-r-1}$} }}  \ar[u] \ar[rrd]&& \vdots  \ar[u]  &&{\text{\huge \textcolor{green}{$ v^{i}_{a_{q+1}-r-1}$} }}  \ar[u] \ar[rrd]&& \vdots  \ar[u]\\
 \vdots\ar[u] &&{\text{\huge \textcolor{green}{$v^{j}_{b_{q+1}-b_{q}} $} }}  \ar[lld] \ar[u]&&\vdots\ar[u] &&{\text{\huge \textcolor{green}{$v^{j}_{b_{q+1}-b_{q}} $} }}  \ar[lld] \ar[u]\\
 {\text{\huge \textcolor{green}{$v^{i}_{a_{q+2}-r-1}$} }} \ar[u]\ar[rrd] & &\vdots\ar[u]&&{\text{\huge \textcolor{green}{$v^{i}_{a_{q+2}-r-1}$} }} \ar[u]\ar[rrd] & &\vdots\ar[u]\\
\vdots \ar[u] & \vdots & {\text{\huge \textcolor{green}{$ v^{j}_{b_{q+2}-b_{q}}$} }} \ar[u] &&\vdots \ar[u] & \vdots & {\text{\huge \textcolor{green}{$ v^{j}_{b_{q+2}-b_{q}}$} }} \ar[u]\\
  {\text{\huge \textcolor{green}{$v^{i}_{a_{q+t}-r-1}$} }}  \ar[rrd] \ar[u] & \vdots   &\vdots \ar[u]&&  {\text{\huge \textcolor{green}{$v^{i}_{a_{q+t}-r-1}$} }}  \ar[rrd] \ar[u] & \vdots   &\vdots \ar[u]\\
\vdots\ar[u] &  &  {\text{\huge {$\mathbf{v^{j}}$} }}  \ar[u]\ar[lld]\ar[llddddd]^{\text{\huge $x$ }}   &&\vdots\ar[u] &  &  {\text{\huge {$\mathbf{v^{j}}$} }}  \ar[u]\ar[lldd]\ar[llddddd]^{\text{\huge $x$ }}  \\
{\text{\huge \textcolor{red}{$v^{i}_{\ell}$} }} \ar[d]  \ar[u] & &  &&{\text{\huge \textcolor{green}{$ v^{i}_{\ell} $} }}   \ar[u]\\
{\text{\huge \textcolor{green}{$v^{i}_{\ell+1}$} }} \ar[rruu] & & &&{\text{\huge \textcolor{red}{$ v^{i}_{\ell+1} $} }}\ar[d]  \ar[u]\\
{\text{\huge \textcolor{green}{$ v^{i}_{\ell+2} $} }} \ar[u]&  &  \vdots \ar[uuu] &&   {\text{\huge \textcolor{green}{$ v^{i}_{\ell+2} $} }} \ar[rruuu] &  &  \vdots \ar[uuu]\\
\vdots \ar[u]&  &   &&\vdots \ar[u]\\
{\text{\huge \textcolor{green}{$v^{i}_{a_{q+t+1}-r}$} }} \ar[rrd] \ar[u]& &  && {\text{\huge \textcolor{green}{$v^{i}_{a_{q+t+1}-r}$} }} \ar[rrd] \ar[u]\\
\vdots\ar[u]& &  {\text{\huge \textcolor{green}{$v^{j}_{b_{q+t+1}-b_{q}}$} }} \ar[uuu]\ar[lld] && \vdots\ar[u]& &  {\text{\huge \textcolor{green}{$v^{j}_{b_{q+t+1}-b_{q}}$} }} \ar[uuu]\ar[lld] \\
{\text{\huge \textcolor{green}{$v^{i}_{a_{q+t+2}-r}$} }} \ar[rrd] \ar[u] &  &\vdots\ar[u]  & \xrightarrow{\text{\huge{$\mu_{v_{k+1}}$}}} &{\text{\huge \textcolor{green}{$v^{i}_{a_{q+t+2}-r}$} }} \ar[rrd] \ar[u] &  &\vdots\ar[u]\\
\vdots\ar[u]&\vdots & {\text{\huge \textcolor{green}{$v^{j}_{b_{q+t+2}-b_{q}}$} }}  \ar[u]   &&\vdots\ar[u]&\vdots & {\text{\huge \textcolor{green}{$v^{j}_{b_{q+t+2}-b_{q}}$} }}  \ar[u]\\
{\text{\huge \textcolor{green}{$ v^{i}_{a_{n-1}-r}$} }} \ar[rrd]  \ar[u]  & \vdots  &\vdots\ar[u] &&{\text{\huge \textcolor{green}{$ v^{i}_{a_{n-1}-r}$} }} \ar[rrd]  \ar[u]  & \vdots  &\vdots\ar[u]\\
\vdots \ar[u]  & &  {\text{\huge \textcolor{green}{$ v^{j}_{b_{n-1}-b_{q}}$} }} \ar[u]\ar[lld] &&\vdots \ar[u]  & &  {\text{\huge \textcolor{green}{$ v^{j}_{b_{n-1}-b_{q}}$} }} \ar[u]\ar[lld] \\
{\text{\huge \textcolor{green}{$ v^{i}_{a_{n}-r}$} }}  \ar[rrdd]  \ar[u]  &  &\vdots\ar[u] &&{\text{\huge \textcolor{green}{$ v^{i}_{a_{n}-r}$} }}  \ar[rrdd]  \ar[u]  &  &\vdots\ar[u]\\
\vdots \ar[u] & &  {\text{\huge \textcolor{green}{$ v^{j}_{b_{n}-b_{q}}$} }} \ar[u] &&\vdots \ar[u] & &  {\text{\huge \textcolor{green}{$ v^{j}_{b_{n}-b_{q}}$} }} \ar[u]\\
 {\text{\huge \textcolor{green}{$ v^{i}_{a_{n+1}-r}$} }} \ar[u] & &  {\text{\huge \textcolor{red}{$v^{j}_{b_{n}-b_{q}+1}$} }} \ar[u] \ar[dd] \ar[lld]_{{\text{\huge {$y $} }}}&&{\text{\huge \textcolor{green}{$ v^{i}_{a_{n+1}-r}$} }} \ar[u] & &  {\text{\huge \textcolor{red}{$v^{j}_{b_{n}-b_{q}+1}$} }} \ar[u] \ar[dd] \ar[lld]_{{\text{\huge {$ y $} }}}\\
{\text{\huge \textcolor{red}{$v^{i}_{a_{n+1}-r+1}$} }} \ar[u]\ar[d]  &  &  &&{\text{\huge \textcolor{red}{$v^{i}_{a_{n+1}-r+1} $} }}\ar[u]\ar[d]\\
\vdots\ar[d] & &  \vdots \ar[dd]&&\vdots\ar[d] & &  \vdots \ar[dd] \\
 {\text{\huge \textcolor{red}{$ v^{i}_{a_{n+1}-a_{q}+1}$} }} \ar[d] \ar[rruuu]_{{\text{\huge {$ y $} }}} &  &&&{\text{\huge \textcolor{red}{$ v^{i}_{a_{n+1}-a_{q}+1}$} }} \ar[d] \ar[rruuu]_{{\text{\huge {$ y $} }}}  \\
 \vdots\ar[d] & & {\text{\huge \textcolor{red}{$  v^{j}_{b_{n}-b_{q-1}+1}$} }}  \ar[llu]\ar[d]&&\vdots\ar[d] & & {\text{\huge \textcolor{red}{$  v^{j}_{b_{n}-b_{q-1}+1}$} }}  \ar[llu]\ar[d]\\
  {\text{\huge \textcolor{red}{$ v^{i}_{a_{n+1}-a_{q-1}+1} $} }} \ar[rru]\ar[d] &  &\vdots \ar[d] && {\text{\huge \textcolor{red}{$ v^{i}_{a_{n+1}-a_{q-1}+1} $} }} \ar[rru]\ar[d] &  &\vdots \ar[d]\\
  \vdots\ar[d] & &  {\text{\huge \textcolor{red}{$v^{j}_{b_{n}-b_{q-2}+1}$} }} \ar[llu]\ar[d] &&\vdots\ar[d] & &  {\text{\huge \textcolor{red}{$v^{j}_{b_{n}-b_{q-2}+1}$} }} \ar[llu]\ar[d] \\
{\text{\huge \textcolor{red}{$ v^{i}_{a_{n+1}-a_{2}+1}$} }} \ar[d] &  &\vdots \ar[d]&&  {\text{\huge \textcolor{red}{$ v^{i}_{a_{n+1}-a_{2}+1}$} }} \ar[d] &  &\vdots \ar[d]\\
\vdots\ar[d] & & {\text{\huge \textcolor{red}{$v^{j}_{b_{n}-b_{1}+1}$} }} \ar[llu]\ar[d] &&\vdots\ar[d] & & {\text{\huge \textcolor{red}{$v^{j}_{b_{n}-b_{1}+1}$} }} \ar[llu]\ar[d] \\
{\text{\huge \textcolor{red}{$v^{i}_{a_{n+1}-a_{1}+1}$} }} \ar[rru]\ar[d] &  &\vdots \ar[d] &&{\text{\huge \textcolor{red}{$v^{i}_{a_{n+1}-a_{1}+1}$} }} \ar[rru]\ar[d] &  &\vdots \ar[d]\\
\vdots\ar[d] & &  {\text{\huge \textcolor{red}{$ v^{j}_{b_{n}}$} }}  &&\vdots\ar[d] & &  {\text{\huge \textcolor{red}{$ v^{j}_{b_{n}}$} }}  \\
{\text{\huge \textcolor{red}{$ v^{i}_{a_{n+1}} $} }}  & & &&{\text{\huge \textcolor{red}{$ v^{i}_{a_{n+1}} $} }}  &
}}
\caption*{(a)}
\end{minipage}\begin{minipage}[b]{0.5\textwidth}
\scalebox{0.35}{\xymatrix@R=10pt@C=0pt{
{\text{\huge \textcolor{green}{$v^{i}_1$} }}  &&   &&{\text{\huge \textcolor{green}{$v^{i}_1$} }}  &&\\
\vdots \ar[u] & & {\text{\huge \textcolor{green}{$ v^{j}_1$} }} && \vdots \ar[u] & & {\text{\huge \textcolor{green}{$ v^{j}_1$} }} \\
{\text{\huge \textcolor{green}{$ v^{i}_{a_{q+1}-r-1}$} }}  \ar[u] \ar[rrd]&& \vdots  \ar[u]  &&{\text{\huge \textcolor{green}{$ v^{i}_{a_{q+1}-r-1}$} }}  \ar[u] \ar[rrd]&& \vdots  \ar[u]\\
 \vdots\ar[u] &&{\text{\huge \textcolor{green}{$v^{j}_{b_{q+1}-b_{q}} $} }}  \ar[lld] \ar[u]&&\vdots\ar[u] &&{\text{\huge \textcolor{green}{$v^{j}_{b_{q+1}-b_{q}} $} }}  \ar[lld] \ar[u]\\
{\text{\huge \textcolor{green}{$v^{i}_{a_{q+2}-r-1}$} }} \ar[u]\ar[rrd] & &\vdots\ar[u]&&{\text{\huge \textcolor{green}{$v^{i}_{a_{q+2}-r-1}$} }} \ar[u]\ar[rrd] & &\vdots\ar[u]\\
\vdots \ar[u] & \vdots & {\text{\huge \textcolor{green}{$ v^{j}_{b_{q+2}-b_{q}}$} }} \ar[u] &&\vdots \ar[u] & \vdots & {\text{\huge \textcolor{green}{$ v^{j}_{b_{q+2}-b_{q}}$} }} \ar[u]\\
{\text{\huge \textcolor{green}{$v^{i}_{a_{q+t}-r-1}$} }}  \ar[rrd] \ar[u] & \vdots   &\vdots \ar[u]&&  {\text{\huge \textcolor{green}{$v^{i}_{a_{q+t}-r-1}$} }}  \ar[rrd] \ar[u] & \vdots   &\vdots \ar[u]\\
\vdots\ar[u] &  &  {\text{\huge {$\mathbf{v^{j}}$} }}  \ar[u]\ar[lld]\ar[llddddd]^{\text{\huge $x$ }}  &&\vdots\ar[u] &  &  {\text{\huge {$\mathbf{v^{j}}$} }}  \ar[u]\ar[lldd]  \\
{\text{\huge \textcolor{red}{$v^{i}_{\ell}$} }} \ar[d]  \ar[u] & &  &&{\text{\huge \textcolor{green}{$ v^{i}_{\ell} $} }}   \ar[u]\\
{\text{\huge \textcolor{green}{$v^{i}_{\ell+1}$} }} \ar[rruu] & & &&{\text{\huge \textcolor{red}{$ v^{i}_{\ell+1} $} }}\ar[ddd]  \ar[u]\\
&  &  \vdots \ar[uuu] &&   &  &  \vdots \ar[uuu]\\
&  &   && \\
{\text{\huge \textcolor{green}{$v^{i}_{a_{q+t+1}-r}$} }} \ar[rrd] \ar[uuu]& &  && {\text{\huge \textcolor{green}{$v^{i}_{a_{q+t+1}-r}$} }} \ar[rrd]\ar[rruuuuu]_{\text{\huge $1-x$ }}\\
\vdots\ar[u]& &  {\text{\huge \textcolor{green}{$v^{j}_{b_{q+t+1}-b_{q}}$} }} \ar[uuu]\ar[lld] && \vdots\ar[u]& &  {\text{\huge \textcolor{green}{$v^{j}_{b_{q+t+1}-b_{q}}$} }} \ar[uuu]\ar[lld] \\
{\text{\huge \textcolor{green}{$ v^{i}_{a_{q+t+2}-r} $} }}\ar[rrd] \ar[u] &  &\vdots\ar[u]  &\xrightarrow{\text{\huge{$\mu_{v_{k+1}}$}}} &{\text{\huge \textcolor{green}{$v^{i}_{a_{q+t+2}-r}$} }} \ar[rrd] \ar[u] &  &\vdots\ar[u]\\
\vdots\ar[u]&\vdots & {\text{\huge \textcolor{green}{$v^{j}_{b_{q+t+2}-b_{q}}$} }}  \ar[u]   &&\vdots\ar[u]&\vdots & {\text{\huge \textcolor{green}{$v^{j}_{b_{q+t+2}-b_{q}}$} }}  \ar[u]\\
{\text{\huge \textcolor{green}{$ v^{i}_{a_{n-1}-r}$} }} \ar[rrd]  \ar[u]  & \vdots  &\vdots\ar[u] &&{\text{\huge \textcolor{green}{$ v^{i}_{a_{n-1}-r}$} }} \ar[rrd]  \ar[u]  & \vdots  &\vdots\ar[u]\\
\vdots \ar[u]  & &  {\text{\huge \textcolor{green}{$ v^{j}_{b_{n-1}-b_{q}}$} }} \ar[u]\ar[lld] &&\vdots \ar[u]  & &  {\text{\huge \textcolor{green}{$ v^{j}_{b_{n-1}-b_{q}}$} }} \ar[u]\ar[lld] \\
{\text{\huge \textcolor{green}{$ v^{i}_{a_{n}-r}$} }}  \ar[rrdd]  \ar[u]  &  &\vdots\ar[u] &&{\text{\huge \textcolor{green}{$ v^{i}_{a_{n}-r}$} }}  \ar[rrdd]  \ar[u]  &  &\vdots\ar[u]\\
\vdots \ar[u] & &  {\text{\huge \textcolor{green}{$ v^{j}_{b_{n}-b_{q}}$} }} \ar[u] &&\vdots \ar[u] & &  {\text{\huge \textcolor{green}{$ v^{j}_{b_{n}-b_{q}}$} }} \ar[u]\\
 {\text{\huge \textcolor{green}{$ v^{i}_{a_{n+1}-r}$} }} \ar[u] & &  {\text{\huge \textcolor{red}{$v^{j}_{b_{n}-b_{q}+1}$} }} \ar[u] \ar[dd] \ar[lld]_{{\text{\huge {$ y$} }}}&&{\text{\huge \textcolor{green}{$ v^{i}_{a_{n+1}-r}$} }} \ar[u] & &  {\text{\huge \textcolor{red}{$v^{j}_{b_{n}-b_{q}+1}$} }} \ar[u] \ar[dd] \ar[lld]_{{\text{\huge {$y $} }}}\\
{\text{\huge \textcolor{red}{$v^{i}_{a_{n+1}-r+1} $} }}\ar[u]\ar[d]  &  &  &&{\text{\huge \textcolor{red}{$v^{i}_{a_{n+1}-r+1} $} }}\ar[u]\ar[d]\\
\vdots\ar[d] & &  \vdots \ar[dd]&&\vdots\ar[d] & &  \vdots \ar[dd] \\
 {\text{\huge \textcolor{red}{$ v^{i}_{a_{n+1}-a_{q}+1}$} }} \ar[d] \ar[rruuu]_{{\text{\huge {$y $} }}} &  &&&{\text{\huge \textcolor{red}{$ v^{i}_{a_{n+1}-a_{q}+1}$} }} \ar[d] \ar[rruuu]_{{\text{\huge {$y $} }}}  \\
 \vdots\ar[d] & & {\text{\huge \textcolor{red}{$  v^{j}_{b_{n}-b_{q-1}+1}$} }}  \ar[llu]\ar[d]&&\vdots\ar[d] & & {\text{\huge \textcolor{red}{$v^{j}_{b_{n}-b_{q-1}+1}$} }}  \ar[llu]\ar[d]\\
{\text{\huge \textcolor{red}{$ v^{i}_{a_{n+1}-a_{q-1}+1} $} }} \ar[rru]\ar[d] &  &\vdots \ar[d] && {\text{\huge \textcolor{red}{$ v^{i}_{a_{n+1}-a_{q-1}+1} $} }} \ar[rru]\ar[d] &  &\vdots \ar[d]\\
\vdots\ar[d] & &  {\text{\huge \textcolor{red}{$v^{j}_{b_{n}-b_{q-2}+1}$} }} \ar[llu]\ar[d] &&\vdots\ar[d] & &  {\text{\huge \textcolor{red}{$v^{j}_{b_{n}-b_{q-2}+1}$} }} \ar[llu]\ar[d] \\
{\text{\huge \textcolor{red}{$ v^{i}_{a_{n+1}-a_{2}+1}$} }} \ar[d] &  &\vdots \ar[d]&&  {\text{\huge \textcolor{red}{$ v^{i}_{a_{n+1}-a_{2}+1}$} }} \ar[d] &  &\vdots \ar[d]\\
\vdots\ar[d] & & {\text{\huge \textcolor{red}{$v^{j}_{b_{n}-b_{1}+1}$} }} \ar[llu]\ar[d] &&\vdots\ar[d] & & {\text{\huge \textcolor{red}{$v^{j}_{b_{n}-b_{1}+1}$} }} \ar[llu]\ar[d] \\
{\text{\huge \textcolor{red}{$v^{i}_{a_{n+1}-a_{1}+1}$} }} \ar[rru]\ar[d] &  &\vdots \ar[d] &&{\text{\huge \textcolor{red}{$v^{i}_{a_{n+1}-a_{1}+1}$} }} \ar[rru]\ar[d] &  &\vdots \ar[d]\\
\vdots\ar[d] & &  {\text{\huge \textcolor{red}{$ v^{j}_{b_{n}}$} }}  &&\vdots\ar[d] & &  {\text{\huge \textcolor{red}{$ v^{j}_{b_{n}}$} }}  \\
{\text{\huge \textcolor{red}{$ v^{i}_{a_{n+1}} $} }}  & & &&{\text{\huge \textcolor{red}{$ v^{i}_{a_{n+1}} $} }}  &
}}
\caption*{(b)}
\end{minipage}}
\caption{(a) is the mutation at $v_{k+1}$ in $\mu_{v_k}\circ \mu_{v_{k-1}} \circ \dots \circ \mu_{v_1}(\widehat{Q^2} )$, where $\ell+1<a_{q+t+1}-r-1$, $v_{k+1}=v^{i}_{\ell+1}$; (b) is the mutation at $v_{k+1}$ in $\mu_{v_k}\circ \mu_{v_{k-1}} \circ \dots \circ \mu_{v_1}(\widehat{Q^2} )$, where $\ell+1=a_{q+t+1}-r-1$, $v_{k+1}=v^{i}_{\ell+1}$. }\label{claim figure at v k+1 (a)}
\end{figure}

\begin{remark}\label{gamma is the quiver (a)}

When $Q^2$ is the quiver shown in Figure \ref{A Quiver of two vertical chains}(a). Similar with Lemma \ref{lemma of two vertical chains}, we can prove that the mutation sequence
\begin{gather}
\begin{align*}
\underline{\mu}:=&\underline{\mu}_{ v^{j}_{b_{n+1}} } \circ\underline{\mu}_{ v^{j}_{b_{n+1}-1} } \circ \cdots \circ \underline{\mu}_{ v^{j}_{b_{n}+1} }\circ
\underline{\mu}_{  v^{i}_{a_{n+1}} } \circ \underline{\mu}_{  v^{i}_{a_{n+1}-1} } \circ \cdots \circ \underline{\mu}_{ v^{i}_{a_{n}+1} }\circ
\underline{\mu}_{  v^{j}_{b_{n}} } \circ \underline{\mu}_{  v^{j}_{b_{n}-1} } \circ \cdots \circ \underline{\mu}_{  v^{j}_{b_{n-1}+1} }  \\
&\circ\underline{\mu}_{  v^{i}_{a_{n}} } \circ \underline{\mu}_{  v^{i}_{a_{n}-1} } \circ \cdots  \circ \underline{\mu}_{  v^{i}_{a_{n-1}+1} }\circ
 \cdots \circ \underline{\mu}_{  v^{j}_{b_{2}} }\circ \underline{\mu}_{  v^{j}_{b_{2}-1} } \circ\cdots \circ \underline{\mu}_{  v^{j}_{b_{1}+1} } \circ
\underline{\mu}_{  v^{i}_{a_2} }\circ \underline{\mu}_{  v^{i}_{a_2-1} } \circ \cdots \circ \underline{\mu}_{  v^{i}_{a_1+1} }\\
& \circ\underline{\mu}_{  v^{j}_{b_{1}} }\circ \underline{\mu}_{  v^{j}_{b_{1}-1}}\circ \cdots \circ  \underline{\mu}_{  v^{j}_{1} }\circ \underline{\mu}_{  v^{i}_{a_1} }\circ \underline{\mu}_{  v^{i}_{a_1-1} }\circ  \cdots \circ \underline{\mu}_{  v^{i}_{1} }
\end{align*}
\end{gather}
is a maximal green sequence of $Q^2$, and the arrows that are incident to the vertex about to perform mutation satisfy conditions (1), (2) and (3) in the proof of Lemma \ref{lemma of two vertical chains}.
\end{remark}

\subsection{Maximal green sequences for $\mathcal{Q}^N$ quivers}\label{Subsection: The proof of main theorem}

Let
\[
\{v^{i_1}_{a_1}, v^{i_2}_{a_2},\dots,v^{i_n}_{a_n}\}
\]
be the vertex set of $Q^N$ such that, for $1\leq j\leq n-1$, $v^{i_{j+1}}_{a_{j+1}} \ngtr v^{i_j}_{a_j}$ with respect to the partial order defined in Definition \ref{definition: partial order}. Now we are ready for the main result of this paper.

\begin{theorem}\label{Main Theorem: maximal green sequences}
The mutation sequence
\[
\underline{\mu}=\underline{\mu}_{v^{i_n}_{a_n}}\circ \underline{\mu}_{v^{i_{n-1}}_{a_{n-1}}} \circ \cdots \circ \underline{\mu}_{v^{i_1}_{a_1}}
\]
 is a maximal green sequence of $Q^N$.
\end{theorem}

\begin{proof}
Assume
\[
\underline{\mu}= \mu_{v_s}\circ \mu_{v_{s-1}} \circ \dots \circ \mu_{v_1}.
\]
We claim that, for $1 \leq \ell \leq s$, if we do not consider the frozen vertices, then the arrows that are incident to the vertex $v_\ell$ in the quiver $\mu_{v_{\ell-1}}  \circ \mu_{v_{\ell-2}} \circ \cdots \circ \mu_{v_1} (\widehat{Q^N})$ form a full subquiver of the quiver
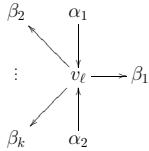
\[
 \scalebox{0.52}{\xymatrix{
{\text{\large$\beta_2$}} & {\text{\large$\alpha_1$}}   \ar[d]&\\
\vdots& {\text{\large$v_{\ell}$}}  \ar[r]\ar[lu]\ar[ld] &  {\text{\large$\beta_1$}} \\
 {\text{\large$\beta_k$}} & {\text{\large$\alpha_2$}} \ar[u]&
}}
\]
such that $\alpha_2 \rightarrow v_{\ell} \rightarrow \alpha_1$ is a full subquiver of one of vertical chains of $Q^N$, and $\beta_1,\beta_2,\dots,\beta_k$ are vertices of other vertical chains. That is,
\begin{itemize}
\item[(1)] the arrows connecting the vertices that belong to the same vertical chain as $v_{\ell}$ to $v_{\ell}$ point toward $v_{\ell}$;
\item[(2)] the arrows connecting the vertices of other vertical chains to $v_{\ell}$ point away $v_{\ell}$.
\end{itemize}
Moreover, every mutable vertex is a green vertex, so we have
\begin{itemize}
\item[(3)]  in the quiver $\mu_{v_{\ell-1}}  \circ \mu_{v_{\ell-2}} \circ \cdots \circ \mu_{v_1} (\widehat{Q^N})$, the arrows connecting the frozen vertices to $v_\ell$ point away $v_\ell$.
\end{itemize}

We prove our claim by induction on the number $N$ of the vertical chains in $Q^N$. The case $N=1$ follows from Remark \ref{The case of one vertical chain}. Assume that our result holds for $N=n$, $n>1$, we prove it for $N=n+1$. According to Definition \ref{subsection: the definition of GHL-quiver}, it can be confirmed that there definitely exists a vertical chain $V$ which is exactly connected to one of the vertical chains in $Q^N$. We denote the portion obtained by removing the vertices in $V_0$ from $Q^N$, along with the arrows connecting to these vertices, as $Q^{N-1}$. Therefore, $Q^N$ can be viewed as being obtained by connecting a vertical chain of $Q^{N-1}$ to the vertical chain $V$, denote this vertical chain of $Q^{N-1}$ by $V'$. Then our theorem holds for $Q^{N-1}$ by induction. Since the arrows connecting $V$ and $V'$ must be in the form shown in Figure \ref{A Quiver of two vertical chains}, our claim follows from Lemma \ref{lemma of two vertical chains}, Remark \ref{gamma is the quiver (a)}, and the proof of Lemma \ref{lemma of two vertical chains}. Indeed, by the rules of quiver mutations, when we mutate any vertex in $\underline{\mu}$, the arrows connecting $V$ and $V'$ will not influence the arrows connecting $V$ and other vertical chains of $Q^{N-1}$, and vice versa.

Let $Q^1:=v^{i}_1 \leftarrow v^{i}_2 \leftarrow \dots \leftarrow v^{i}_{k_i}$ be a vertical chain of $Q^N$ with the vertex set $Q^1_0=\{v^{i}_{\ell}\mid\ell=1,2,\dots,k_i\}$. It follows from (1), (2) and (3) that the full subquiver of
\[
\mu_{v_s} \circ \mu_{v_{s-1}}\circ \dots \circ \mu_{v_1} (\widehat{Q^N})
\]
with the vertex set $\widehat{Q^1_0}=\{v^{i}_{\ell},{v^{i}_{\ell}}'\mid\ell=1,2,\dots,k_i\}$ is the same as the quiver
\[
\mu_{v_s} \circ \mu_{v_{s-1}}\circ \dots \circ \mu_{v_1}\vert_{Q^1_0}(\widehat{Q^1}),
\]
where the mutation sequence $\mu_{v_s} \circ \mu_{v_{s-1}}\circ \dots \circ \mu_{v_1}\vert_{Q^1_0}$ is obtained from the mutation sequence $\mu_{v_s} \circ \mu_{v_{s-1}}\circ \dots \circ \mu_{v_1}$ by deleting the vertices that do not belong to $Q^1_0$. Moreover, there is no arrow connecting the vertices in $\{v^{i}_{\ell} \mid\ell=1,2,\dots,k_i\}$ to the frozen vertices of $\widehat{Q^N_0}\setminus \{{v^{i}_{\ell}}'\mid\ell=1,2,\dots,k_i\}$, where $\widehat{Q^N_0}$ is the vertex set of $\widehat{Q^N}$. Since
\[
\mu_{v_s} \circ \mu_{v_{s-1}}\circ \dots \circ \mu_{v_1}\vert_{Q^1_0}=\underline{\mu}_{v^{i}_{k_i}}\circ\underline{\mu}_{v^{i}_{k_i-1}}\circ\dots\circ
\underline{\mu}_{v^{i}_{1}}
\]
is a maximal green sequence of $Q^1$, our result holds.
\end{proof}

\section{Special cases of $\mathcal{Q}^N$ quivers}\label{Section: Applications}
Many quivers that we are familiar with are $\mathcal{Q}^N$ quivers. In this section, We prove that any finite connected full subquiver of the quivers defined by Hernandez and Leclerc, arising in monoidal categorifications of cluster algebras, is a special case of $\mathcal{Q}^N$ quivers. Moreover, we prove that the trees of oriented cycles introduced by Garver and Musiker and the quivers in $\mu^A$ and $\mu^D$ can also be recognized as special cases of $\mathcal{Q}^N$ quivers.

\subsection{Quivers arising from monoidal categorifications}

Let $\mathfrak{g}$ be a complex simple finite-dimensional Lie algebra of rank $n$. It is well-known that $\mathfrak{g}$ is completely classified by the Dynkin diagrams. In this paper, we use the same labellings of Dynkin diagrams as the ones in \cite[Section 4.3]{K1990}. Let $C=(c_{ij})_{i, j\in I}$ be the Cartan matrix of $\mathfrak{g}$, where $I=\{1,2,\ldots,n\}$. There is a diagonal matrix $D=\mathrm{diag}(d_i \mid i\in I)$ with positive entries $d_i$ such that $B=DC=(b_{ij})_{i, j\in I}$ is a symmetric matrix.  We choose $D$ such that $\min\{d_i \mid i\in I\}=1$. Let $r= \max\{d_i \mid i\in I\}$. Then
\[
r=\begin{cases}
1 \quad&\text{if $C$ is of type $A_n$, $D_n$, $E_6$, $E_7$ or $E_8$},\\
2 \quad&\text{if $C$ is of type $B_n$, $C_n$ or $F_4$},\\
3 \quad&\text{if $C$ is of type $G_2$}.
\end{cases}
\]

Following \cite{HL16}, let $\tilde{G}$ be an infinite quiver with vertex set $\widetilde{V} =I\times \mathbb{Z}$ and arrows $((i,r)\rightarrow (j,s)) \iff (b_{ij}\neq0$, $s-d_j=r-d_i+b_{ij})$.  According to  \cite{HL16}, we know that $\tilde{G}$ has two isomorphic connected components, which are identical if we disregard the vertex labels. Here, we ignore the vertex labels and denote one of the connected components as $G$.  An example of the quiver $G$ of type $B_2$ and a finite full subquiver $\bar{G}$ of $G$ are given in Figure \ref{Figure:Examples of Gamma and Gamma xi in type B2}.
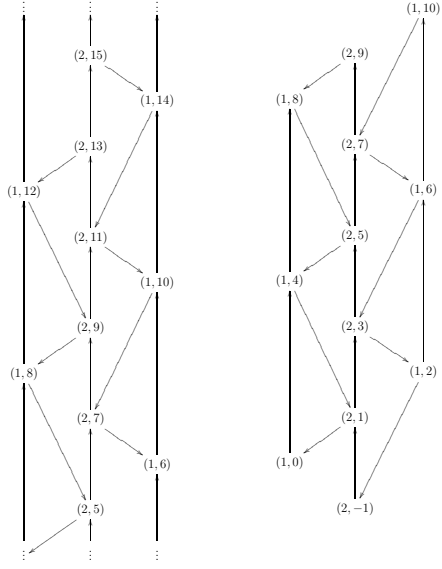
\begin{figure}
\centering
\scalebox{0.4}{
\begin{minipage}[t]{0.5\textwidth}
\xymatrix{
\vdots &\vdots & \vdots\\
& (2,15)\ar[rd] \ar[u]& \\
& & (1,14)\ar[uu]\ar[lddd] \\
& (2,13)\ar[ld]\ar[uu]  & \\
 (1,12) \ar[rddd]\ar[uuuu] &   & \\
  &(2,11) \ar[uu] \ar[rd]& \\
& &(1,10)\ar[lddd]\ar[uuuu]  \\
 &(2,9) \ar[ld]\ar[uu] \\
(1,8)\ar[rddd]\ar[uuuu]   & &    \\
 &(2,7)\ar[uu]  \ar[rd]  &   \\
  & & (1,6) \ar[uuuu]   \\
  & (2,5) \ar[ld]\ar[uu] &   \\
 \vdots\ar[uuuu] & \vdots \ar[u] & \vdots \ar[uu]}
\end{minipage}
\hskip 3cm
\begin{minipage}[t]{0.5\textwidth}
\xymatrix{
&   & (1,10)  \ar[lddd]  \\
&   (2,9) \ar[ld] & \\
(1,8) \ar[rddd]  & &  \\
 & (2,7)\ar[uu] \ar[rd]& &  \\
 & & (1,6) \ar[uuuu] \ar[lddd]  \\
  &(2,5) \ar[ld]\ar[uu]& &  \\
 (1,4) \ar[rddd]  \ar[uuuu] &&  \\
 &(2,3)\ar[uu] \ar[rd]& &  \\
&   & (1,2) \ar[uuuu] \ar[lddd]  \\
&  (2,1) \ar[ld]\ar[uu]& &  \\
(1,0)  \ar[uuuu] &&  \\
 & (2,-1) \ar[uu]& &  \\
}
\end{minipage}}
\caption{In type $B_2$, quiver $G$ (left) and a finite full subquiver $\bar{G}$ of $G$ (right).} \label{Figure:Examples of Gamma and Gamma xi in type B2}
\end{figure}

\begin{proposition}\label{HL quiver are GHL quiver}
 Let $\bar{G}$ be a finite connected full subquiver of $G$. Then $\bar{G}$ is a $\mathcal{Q}^N$ quiver.
\end{proposition}

\begin{proof}
In $\bar{G}$, we can view
\[
(i,a)\leftarrow (i,a-2d_i)\leftarrow \dots\leftarrow (i,a-2kd_{i})
\]
as a vertical chain, where $i \in I$, $(i,a-2\ell d_i)\in \bar{G}_0$ for $\ell=0,1,\dots,k$, and $(i,a+2d_i), (i,a-2(k+1)d_{i})\notin \bar{G}_0$, which satisfies the conditions of Definition \ref{the definition of generalized HL-quiver}(1).
The arrows connecting two vertical chains exactly form a path
\begin{align*}
&(i, a)\rightarrow (j,a +b_{ij}+d_{j}-d_{i})\rightarrow (i,a +2b_{ij})\rightarrow (j,a +3b_{ij}+d_{j}-d_{i})\rightarrow (i,a +4b_{ij}) \rightarrow \dots,
\end{align*}
where  $i, j\in I$, the vertices $i$ and $j$ are adjacent in Dynkin diagram, $a\in\mathbb{Z}$, which satisfy the conditions of Definition \ref{the definition of generalized HL-quiver}(2). It follows from the fact that Dynkin diagrams do not have cycles that the arrows connecting the vertical chains satisfy the conditions of Definition \ref{the definition of generalized HL-quiver}(3).
\end{proof}

\begin{example}
Let $Q^3$ be the quiver $\bar{G}$ shown in the right diagram of Figure \ref{Figure:Examples of Gamma and Gamma xi in type B2}. It follows that
\begin{gather}
\begin{align*}
(1,10)>(2,9)>(1,8)>(2,7)>(1,6)>(2,5)>(1,4)>(2,3)>(1,2)>(2,1)>(1,0)>(2,-1),
\end{align*}
\end{gather}
with respect to the partial order defined in Definition \ref{definition: partial order}. By Theorem \ref{Main Theorem: maximal green sequences}, we obtain a maximal green sequence of $Q^3$ as follows:
\begin{align*}
&\underline{\mu}_{(2,-1)}\circ\underline{\mu}_{(1,0)}\circ\underline{\mu}_{(2,1)}\circ\underline{\mu}_{(1,2)}\circ\underline{\mu}_{(2,3)}\circ\underline{\mu}_{(1,4)}\circ\underline{\mu}_{(2,5)}\circ\underline{\mu}_{(1,6)}\circ\underline{\mu}_{(2,7)}\circ\underline{\mu}_{(1,8)}\circ\underline{\mu}_{(2,9)}\circ\underline{\mu}_{(1,10)}\\
=&\mu_{(2,9)}\circ \mu_{(1,8)}\circ \mu_{(2,7)} \circ\mu_{(2,9)} \circ\mu_{(1,10)}\circ\mu_{(2,5)} \circ\mu_{(2,7)} \circ\mu_{(2,9)} \circ\mu_{(1,4)} \circ\mu_{(1,8)} \circ\mu_{(2,3)} \circ\\
&\mu_{(2,5)} \circ\mu_{(2,7)} \circ\mu_{(2,9)} \circ\mu_{(1,6)}\circ\mu_{(1,10)}\circ\mu_{(2,1)} \circ\mu_{(2,3)} \circ\mu_{(2,5)}\circ \mu_{(2,7)}\circ \mu_{(2,9)} \circ\mu_{(1,0)}\circ\\
&\mu_{(1,4)} \circ\mu_{(1,8)} \circ\mu_{(2,-1)} \circ\mu_{(2,1)} \circ\mu_{(2,3)} \circ\mu_{(2,5)} \circ\mu_{(2,7)} \circ\mu_{(2,9)} \circ\mu_{(1,2)}\circ\mu_{(1,6)}\circ\mu_{(1,10)}.
\end{align*}
\end{example}

\subsection{Trees of oriented cycles} \label{subsection: trees of cycles}
Garver and Musiker introduced the concept of trees of $3$-cycles, and generalized it to trees of oriented cycles in \cite{GM17}. For an irreducible quiver $Q$ that is mutation equivalent to an orientation of a type A Dynkin diagram, it can be obtained by gluing together a finite number $m$ of oriented $3$-cycles $\{S_i\mid i=1,2,\dots,m\}$ such that each oriented $3$-cycle shares a vertex with at most three other oriented $3$-cycles and the number of cycles in the underlying graph of $Q$ is $m$. Garver and Musiker proved that $Q$ is equivalent to a binary tree with vertex set $\{S_i\mid i=1,2,\dots,m\}$, where $S_i$ is connected to $S_j$ by an edge if and only if $S_i$ and $S_j$ share a vertex in $Q$, and $Q$ is called a \textit{tree of $3$-cycles}.

An irreducible quiver is called a \textit{tree of oriented cycles} if $Q$ is obtained by gluing together a finite number $m$ of oriented cycles $\{S_i\mid i=1,2,\dots,m\}$ such that each oriented 
$k$-cycle shares a vertex with at most $k$ other oriented cycles and the number of cycles in the underlying graph of $Q$ is $m$, where each cycle $S_i$ has length at least $3$. Note that each vertex in $Q$ can appear in at most two cycles, and every cycle in $Q$ is an oriented cycle.

\begin{proposition}\label{prop:tree of oriented cycles}
Trees of oriented cycles are $\mathcal{Q}^N$ quivers.
\end{proposition}

\begin{proof}
Let $Q$ be a tree of oriented cycles, and let $Q'=v_1\rightarrow v_2 \rightarrow \dots \rightarrow v_\ell \rightarrow v_1$ be an $\ell$-cycle in $Q$. Then we put $v_1\rightarrow v_2 \rightarrow \dots \rightarrow v_{\ell-1}$ in a vertical chain of length $(\ell-1)$, and put $v_\ell$ in another chain. Assume that $Q''$ is another cycle in $Q$ and $Q'_0\cap Q''_0=v$. Then we put the vertices $Q''_0\backslash v$ in a vertical chain alone, and keep $v$ stay in the original vertical chain. Repeat the above process, $Q$ can always be seen as a $\mathcal{Q}^N$ quiver.
\end{proof}

\begin{example}
An example of tree of oriented cycles is given  in Figure \ref{example: a tree of oriented cycles}.
\begin{figure}
\centering
\scalebox{0.5}{\xymatrix{
&  & 6\ar[rd] & 1 \ar[rdd] &  \\
&  & 7\ar[u] & 2 \ar[u]\ar[ld]& &9 \ar[ld] \\
&11 \ar[rrd] & 8 \ar[u] & & 5\ar[ldd]\ar[rd] & \\
15  \ar[rd]&12 \ar[u] &   &3\ar[lldd]\ar[uu]&   & 10\ar[uu]\\
16 \ar[u] &13 \ar[ld] \ar[u]&   &4\ar[u]&   & \\
17 \ar[u] &14 \ar[u] &   & &   & \\
}}
\caption{A tree of oriented cycles.} \label{example: a tree of oriented cycles}
\end{figure}
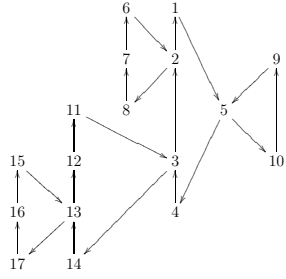
\end{example}

Theorem \ref{Main Theorem: maximal green sequences} and Proposition \ref{prop:tree of oriented cycles} provide a solution for the following open problem proposed by Garver and Musiker.

\begin{problem}\cite[Problem 8.3]{GM17}\label{Open problem1}
Find a construction of maximal green sequences for quivers that are trees of oriented cycles.
\end{problem}

\begin{example}
Let $Q^6$ be the quiver shown in Figure \ref{example: a tree of oriented cycles}. It follows that
\begin{gather}
\begin{align*}
&15>11>12>13>16>17>14, \quad 11>1>2>3>12>13>14>4, \\ &6>1>2>7>8>3>4, \quad 1>5>2>3>4, \quad 9>5>10,
\end{align*}
\end{gather}
with respect to the partial order defined in Section \ref{section: partial order}. By Theorem \ref{Main Theorem: maximal green sequences}, we obtain a maximal green sequence of $Q^6$ as follows:
\begin{align*}
&\underline{\mu}_{4}\circ\underline{\mu}_{14}\circ\underline{\mu}_{17}\circ\underline{\mu}_{16}\circ\underline{\mu}_{13}\circ\underline{\mu}_{12}\circ\underline{\mu}_{3}\circ\underline{\mu}_{8}\circ\underline{\mu}_{7}\circ\underline{\mu}_{2}\circ\underline{\mu}_{10}\circ\underline{\mu}_{5}\circ\underline{\mu}_{9}\circ \underline{\mu}_{1}\circ\underline{\mu}_{ 6}\circ\underline{\mu}_{ 11}\circ\underline{\mu}_{15}\\
=& \mu_{ 1}\circ\mu_{11 }\circ\mu_{ 15} \circ\mu_{ 16}\circ\mu_{ 15}\circ \mu_{ 12} \circ\mu_{11 }\circ\mu_{13 }\circ\mu_{ 12}\circ\mu_{11 } \circ\mu_{2 }\circ\mu_{1 }\circ \mu_{6 } \circ\mu_{ 7}\circ\mu_{6 }\circ\mu_{ 3}\circ \mu_{2 } \\
 &\circ\mu_{ 1}\circ\mu_{9 }\circ\mu_{5 }\circ\mu_{ 10} \circ\mu_{ 9}\circ\mu_{4 }\circ\mu_{ 3}\circ\mu_{2 }  \circ\mu_{ 1}\circ\mu_{8 }\circ\mu_{ 7}\circ \mu_{ 6} \circ\mu_{ 14}\circ\mu_{13 }\circ\mu_{ 12}\circ\mu_{11 } \circ\mu_{17}\\
 &\circ\mu_{16 }\circ\mu_{15}.
\end{align*}
\end{example}

In the following, we study a class of irreducible quivers whose cycles are all oriented, which is a generalization of trees of oriented cycles. It turns out that these irreducible quivers are $\mathcal{Q}^N$ quivers.

\begin{lemma}
Let $Q$ be an irreducible quiver such that each cycle in $Q$ is an oriented cycle. Then any two cycles in $Q$ can have at most one intersection.
\end{lemma}

\begin{proof}
Assume that there are two oriented cycles $Q'$ and $Q''$ that have $k$ intersections in $Q$ with $k\geq2$. Let $v_1$ and $v_2$ be two vertices such that $v_1 - a_1 - a_2 - \dots - a_{k_1} - v_2$ forms a path $p_1$ in the underlying graph of $Q'$, and $v_1 - b_1 - b_2 - \dots - b_{k_2} - v_2$ forms a path $p_2$ in the underlying graph of $Q''$, and moreover, $p_1$ (respectively, $p_2$) and $Q''$ (respectively, $Q'$) do not  intersect at any point other than $v_1$ and $v_2$.

Now, paths $p_1$ and $p_2$ form a cycle. If this cycle is not oriented, then our result follows. If it is oriented, there must be a path $p'_1$ connecting $v_1$ to $v_2$ in $Q'$ such that $p'_1$ and $p_2$ form a cycle. Furthermore, this cycle is not oriented; otherwise, it contradicts the assertion that $Q'$ is an oriented cycle.
\end{proof}

\begin{proposition}\label{Prop:irreducible quivers contains oriented cycles}
Let $Q$ be an irreducible quiver such that each cycle in $Q$ is an oriented cycle. Then $Q$ is a $\mathcal{Q}^N$ quiver.
\end{proposition}

\begin{proof}
The proof is the same as that of Proposition \ref{prop:tree of oriented cycles}.
\end{proof}

\begin{remark}
If $Q$ is a tree of oriented cycles, it follows from the definition of tree of oriented cycles that each cycle in $Q$ is an oriented cycle. The difference between the quivers described Proposition \ref{Prop:irreducible quivers contains oriented cycles} and trees of oriented cycles lies in the fact that in the quiver described Proposition \ref{Prop:irreducible quivers contains oriented cycles}, any finite number of oriented cycles can intersect at a vertex, whereas in a tree of oriented cycles, at most two oriented cycles intersect at a vertex.
\end{remark}

\begin{example}
An example of the irreducible quiver described in Proposition \ref{Prop:irreducible quivers contains oriented cycles} is as follows
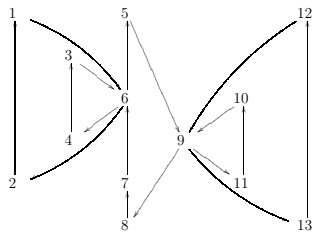
\[
 \scalebox{0.4}{\xymatrix{
{\text{\large { $ 1$ } }} \ar@/^1pc/[rrdd]&  &{\text{\large { $ 5$ } }}\ar[rddd] &   & & {\text{\large { $ 12$ } }}\ar@/_1pc/[llddd] \\
& {\text{\large { $3$ } }}  \ar[rd] &  &   & &  \\
& & {\text{\large { $ 6$ } }} \ar[uu]\ar[ld] \ar@/^1pc/[lldd]&   & {\text{\large { $ 10$ } }}\ar[ld]&  \\
&{\text{\large { $ 4$ } }} \ar[uu] &   &  {\text{\large { $ 9$ } }} \ar@/_1pc/[rrdd] \ar[ldd]\ar[rd]&  &  \\
{\text{\large { $ 2$ } }} \ar[uuuu] & & {\text{\large { $ 7$ } }}\ar[uu]  &   & {\text{\large { $ 11$ } }} \ar[uu]&  \\
& & {\text{\large { $ 8$ } }} \ar[u]&   &  & {\text{\large { $ 13$ } }} \ar[uuuuu]\\
}}
\]
\end{example}

\subsection{Quivers in $\mu^A$ and $\mu^D$}
We use $\mu^{A}$ and $\mu^{D}$ to denote the set of quivers that are mutation equivalent to orientations of type A and type D Dynkin diagrams, respectively. Proposition 2.4 of \cite{BV08} and Theorem 3.1 of \cite{V10} gave explicit descriptions of quivers in $\mu^{A}$ and $\mu^{D}$, respectively. Let us now recall the content of these results.

\begin{lemma}[{\cite[Proposition 2.4]{BV08}}]\label{lemma:type A quiver}
A quiver $Q$ is mutation equivalent to an orientation of a type A Dynkin diagram if and only if $Q$ satisfies the following conditions:
\begin{itemize}
\item all non-trivial cycles are of length $3$ and oriented,
\item a vertex has at most four neighbours,
\item if a vertex has four neighbours, then two of its incident arrows belong to a $3$-cycle, and
the other two arrows belong to another $3$-cycle,
\item if a vertex has precisely three neighbours, then two of its incident arrows belong to a $3$-cycle,
and the third arrow does not belong to any $3$-cycle.
\end{itemize}
\end{lemma}

Let $Q\in\mu^A$ and $v\in Q_0$. Following \cite{V10}, the vertex $v$ is said to be a connecting vertex of $Q$ if $v$ has at most $2$ neighbors, and moreover, if $v$ has $2$ neighbors, then $v$ is in a $3$-cycle.

\begin{lemma}[{\cite[Theorem 3.1]{V10}}] \label{Lemma: Four types of quiver in type D}
A quiver $Q$ is mutation equivalent to an orientation of a type D Dynkin diagram if and only if $Q$ is one of the following types:
\begin{itemize}
\item[Type I:] In quiver $Q$, there are two vertices $a$ and $b$ such that they have a common neighbor $c$ and they do not have other neighbors. Moreover, $Q\backslash \{a, b\}=Q'$, $Q'\in \mu^A$, and $c$ is a connecting vertex of $Q'$, $Q$ is shown in Figure \ref{The quiver that are mutation equivalent to D quiver}(a).

\item[Type II:] Quiver $Q$ has a full subquiver with vertex set $\{a, b, c, d\}$, as shown by the black portion in Figure \ref{The quiver that are mutation equivalent to D quiver}(b). Moreover, $Q\backslash \{a,b,c \rightarrow d\}=Q'\sqcup Q''$, where $Q', Q'' \in \mu^{A}$, and there are no arrows connecting the vertices of $Q'$ to the vertices of $Q''$. Furthermore, $c$ and $d$ are connecting vertices of $Q'$ and $Q''$, respectively, $Q$ is shown in Figure \ref{The quiver that are mutation equivalent to D quiver}(b).

\item[Type III:] Quiver $Q$ has a full subquiver  with vertex set $\{a, b, c, d\}$, as shown by the black portion in Figure \ref{The quiver that are mutation equivalent to D quiver}(c). Moreover, $Q\backslash \{a, b\}=Q' \sqcup Q''$, where $Q', Q'' \in \mu^{A}$, and there are no arrows connecting the vertices of $Q'$ to the vertices of $Q''$. Furthermore, $c$ and $d$ are connecting vertices of $Q'$ and $Q''$, respectively, $Q$ is shown in Figure \ref{The quiver that are mutation equivalent to D quiver}(c).

\item[Type IV:] In quiver $Q$, there exists a full subquiver which is an oriented $k$-cycle, where $k\geq 3$. This oriented $k$-cycle is called a central cycle. For each arrow $\alpha: a \rightarrow b$ in this central cycle, if there is a vertex $c_\alpha$ which is not on the central cycle, such that $a\rightarrow b \rightarrow c_{\alpha} \rightarrow a$ is a full subquiver of $Q$, then such a $3$-cycle is called a \textit{spike}. No additional arrows starting or ending at vertices on the central cycle. Moreover, $Q\backslash \{ \text{vertices in the central cycle and their incident arrows} \}=Q' \sqcup Q''  \sqcup Q'''\dots$, there are no arrows connecting the vertices in $Q'$, $Q''$, $Q'''$, $\dots$ with each other, and $Q'$, $Q''$, $Q'''$, $\dots$ are all in $\mu^A$. Furthermore, $c'$, $c''$, $c'''$, $\dots$ are connecting vertices of $Q'$, $Q''$, $Q'''$, $\dots$, respectively, $Q$ is shown in Figure \ref{The quiver that are mutation equivalent to D quiver}(d).
\end{itemize}
\begin{figure} 
\centering
\begin{minipage}[t]{0.3\textwidth}
\resizebox{\textwidth}{!}{ 
\begin{tikzpicture}
\coordinate (A) at (0,0);
\coordinate (B) at (0,2);
\coordinate (C) at (1,1);
\fill (A) circle (2pt); \fill (B) circle (2pt); \fill (C) circle (2pt);
\draw (A) -- (C) -- (B);
\draw[blue] (1.8,1) ellipse (1.2cm and 0.7cm);
\node at (-0.2,2) {a}; \node at (-0.2,0) {b}; \node at (1.2,1) {c};
\node at (2.2,1) {$\textcolor{blue}{Q'}$};
\end{tikzpicture}
}\caption*{(a)}
\end{minipage}
\hfill
\begin{minipage}[t]{0.48\textwidth}
\resizebox{\textwidth}{!}{ 
\begin{tikzpicture}
\coordinate (D) at (0,1);
\coordinate (B) at (1,2);
\coordinate (C) at (2,1);
\coordinate (A) at (1,0);
\draw[-stealth, shorten >=1mm, shorten <=1mm] (D) -- (C);
\draw[-stealth, shorten >=1mm, shorten <=1mm] (C) -- (B);
\draw[-stealth, shorten >=1mm, shorten <=1mm] (B) -- (D);
\draw[-stealth, shorten >=1mm, shorten <=1mm] (A) -- (D);
\draw[-stealth, shorten >=1mm, shorten <=1mm] (C) -- (A);
\fill (A) circle (2pt) node[below] {a};
\fill (B) circle (2pt) node[above] {b};
\fill (C) circle (2pt) node[right] {d};
\fill (D) circle (2pt) node[left] {c};
\node at (-1.2,1) {$\textcolor{blue}{Q'}$};  \node at (3.2,1) {$\textcolor{blue}{Q''}$};
\draw[blue] (-0.85,1) ellipse (1.2cm and 0.7cm); \draw[blue] (2.85,1) ellipse (1.2cm and 0.7cm);
\end{tikzpicture}
}\caption*{(b)}
\end{minipage}
\vspace{1em} 
\begin{minipage}[t]{0.48\textwidth}
\resizebox{\textwidth}{!}{ 
\begin{tikzpicture}
\coordinate (D) at (0,1);
\coordinate (C) at (1,2);
\coordinate (B) at (2,1);
\coordinate (A) at (1,0);
\draw[-stealth, shorten >=1mm, shorten <=1mm] (A) -- (B);
\draw[-stealth, shorten >=1mm, shorten <=1mm] (B) -- (C);
\draw[-stealth, shorten >=1mm, shorten <=1mm] (C) -- (D);
\draw[-stealth, shorten >=1mm, shorten <=1mm] (D) -- (A);
\fill (A) circle (2pt) node[below] {a};
\fill (C) circle (2pt) node[above] {b};
\fill (B) circle (2pt) node[right] {d};
\fill (D) circle (2pt) node[left] {c};
\node at (-1.2,1) {$\textcolor{blue}{Q'}$};  \node at (3.2,1) {$\textcolor{blue}{Q''}$};
\draw[blue] (-0.85,1) ellipse (1.2cm and 0.7cm); \draw[blue] (2.85,1) ellipse (1.2cm and 0.7cm);
\end{tikzpicture}
}\caption*{(c)}
\end{minipage}
\hfill
\begin{minipage}[t]{0.48\textwidth}
\resizebox{\textwidth}{!}{ 
\begin{tikzpicture}
\coordinate (v1) at (2,0); \coordinate (v2) at (3,0.5);
\coordinate (v3) at (3.5,1.5);   \coordinate (v4) at (4,2.5);
\coordinate (v5) at (3,2.5); \coordinate (v6) at (3,3.5);
\coordinate (v7) at (2,3); \coordinate (v8) at (1,2.5);
\coordinate (v9) at (0.5,1.5); \coordinate (v10) at (0,2.5);
\fill (v1) circle (2pt); \fill (v2) circle (2pt);
\fill (v3) circle (2pt); \fill (v4) circle (2pt);
\fill (v5) circle (2pt); \fill (v6) circle (2pt);
\fill (v7) circle (2pt); \fill (v8) circle (2pt);
\fill (v9) circle (2pt); \fill (v10) circle (2pt);
\draw[-stealth, shorten >=1mm, shorten <=1mm] (v1) -- (v2);
\draw[-stealth, shorten >=1mm, shorten <=1mm] (v2) -- (v3);
\draw[-stealth, shorten >=1mm, shorten <=1mm] (v3) -- (v5);
\draw[-stealth, shorten >=1mm, shorten <=1mm] (v5) -- (v7);
\draw[-stealth, shorten >=1mm, shorten <=1mm] (v7) -- (v8);
\draw[-stealth, shorten >=1mm, shorten <=1mm] (v8) -- (v9);
\draw[-stealth, shorten >=1mm, shorten <=1mm] (v9) -- (v10);
\draw[-stealth, shorten >=1mm, shorten <=1mm] (v10) -- (v8);
\draw[-stealth, shorten >=1mm, shorten <=1mm] (v6) -- (v5);
\draw[-stealth, shorten >=1mm, shorten <=1mm] (v5) -- (v4);
\draw[-stealth, shorten >=1mm, shorten <=1mm] (v4) -- (v3);
\draw[-stealth, shorten >=1mm, shorten <=1mm] (v7) -- (v6);
\draw[dashed] (v1) to[bend left=45] (v9);
\node at (v10) [left] {$c'$};  \node at (v6) [above] {$c''$};  \node at (v4) [right] {$c'''$};
\node at (5.5,2.5){$\textcolor{blue}{Q'''}$};  \node at (3,4.6) {$\textcolor{blue}{Q''}$};  \node at (-1.5,2.5) {$\textcolor{blue}{Q'}$};
\draw[blue] (-0.85,2.5) ellipse (1.2cm and 0.7cm); \draw[blue] (3,4.2) ellipse (0.7cm and 0.9cm);
\draw[blue] (4.85,2.5) ellipse (1.2cm and 0.7cm);
\end{tikzpicture}
}\caption*{(d)}
\end{minipage}\caption{(a), (b), (c), (d) are quivers of Type I, Type II, Type III, Type IV in $\mu^D$, respectively.}\label{The quiver that are mutation equivalent to D quiver}
\end{figure}
\end{lemma}

\begin{proposition}\label{prop:quiver in mu A}
If $Q\in \mu^A$, then $Q$ is a $\mathcal{Q}^N$ quiver.
\end{proposition}

\begin{proof}
By Lemma \ref{lemma:type A quiver}, $Q$ is composed of some $3$-cycles and some vertices that are not in the cycles, as well as the arrows that connect cycles and vertices. If a vertex $v$ is in a cycle of $Q$, then this cycle is an oriented $3$-cycle. Let $Q'=v_1\rightarrow v_2 \rightarrow v_3 \rightarrow v_1$ be a $3$-cycle in $Q$. Then we can put $v_1$, $v_2$ in a vertical chain, and put $v_3$ in another vertical chain. If there exists another $3$-cycle $Q''$ such that $Q'_0\cap Q''_0=v$, then we put $Q''_0 \backslash v$ alone in a vertical chain of length $2$, and keep $v$ stay in original vertical chain; if $v$ does not belong to a cycle of $Q$ and it is adjacent to one of the vertices of  $Q'_0$, then we put $v$ alone in a vertical chain. Repeat the above process, $Q$ can always be seen as a $\mathcal{Q}^N$ quiver.
\end{proof}

\begin{example}
Some examples of the quivers in $\mu^A$ are as follows
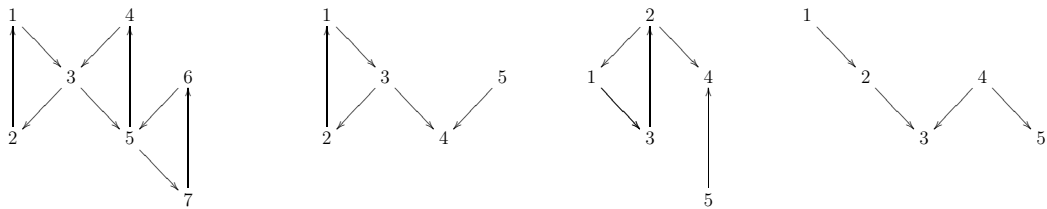
\[
\begin{minipage}[t]{0.22\textwidth}
 \scalebox{0.6}{\xymatrix{
1\ar[rd]& &4 \ar[ld]\\
 &3\ar[ld] \ar[rd]&&6\ar[ld]\\
2\ar[uu]& &5 \ar[uu]\ar[rd] \\
&&&7\ar[uu]
 }}
\end{minipage}
\begin{minipage}[t]{0.22\textwidth}
 \scalebox{0.6}{\xymatrix{
&1\ar[rd]& & \\
& &3\ar[ld] \ar[rd]&&5\ar[ld]\\
&2\ar[uu]& &4    \\
 }}
\end{minipage}
\begin{minipage}[t]{0.22\textwidth}
 \scalebox{0.6}{\xymatrix{
&&2\ar[ld]\ar[rd]& & \\
& 1\ar[rd] \ar[rd]&&4\\
&&3\ar[uu]& &   \\
&&&  5\ar[uu]
 }}
\end{minipage}
\begin{minipage}[t]{0.14\textwidth}
 \scalebox{0.6}{\xymatrix{
1\ar[rd]& \\
&2\ar[rd]&&4\ar[ld]\ar[rd]\\
 & &3 &&5\\
 }}
\end{minipage}
\]
\end{example}

\begin{proposition}\label{theorem: a quiver of Type D is a generalized HL-quiver}
If $Q\in \mu^D$, then $Q$ is a $\mathcal{Q}^N$ quiver.
\end{proposition}

\begin{proof}
Let $Q\in \mu^{D}$. We need to consider $4$ cases.

{\bf Case 1.} If $Q$ is of Type I of Lemma \ref{Lemma: Four types of quiver in type D}. Since $Q' \in\mu^{A}$, we can rearrange the vertices of $Q'$ such that $Q'$ is a $\mathcal{Q}^N$ quiver, as shown in Proposition \ref{prop:quiver in mu A}. Next, we just need to put vertices $a$ and $b$ in two vertical chains of length $1$ of $Q$, respectively. Then $Q$ can be viewed as a $\mathcal{Q}^N$ quiver.

{\bf Case 2.}  If $Q$ is of Type II of Lemma \ref{Lemma: Four types of quiver in type D}. Since $Q', Q'' \in\mu^{A}$, by Proposition \ref{prop:quiver in mu A}, we can rearrange the vertices of $Q'$ and $Q''$ such that $Q'$ and $Q''$ are $\mathcal{Q}^N$ quivers, and $c$, $d$ lie in two vertical chains of length $1$ of $Q'$ and $Q''$, respectively. Next, we just need to put vertices $a$ and $b$ in two separate vertical chains alone, and put $c \rightarrow d$ in a vertical chain alone in $Q$. Then $Q$ can be viewed as a $\mathcal{Q}^N$ quiver.

{\bf Case 3.} If $Q$ is of Type III of Lemma \ref{Lemma: Four types of quiver in type D}. Since $Q', Q'' \in\mu^{A}$, by Proposition \ref{prop:quiver in mu A}, we can rearrange the vertices of $Q'$ and $Q''$ such that $Q'$ and $Q''$ are $\mathcal{Q}^N$ quivers, and $c$, $d$ lie in two vertical chains of length $1$ of $Q'$, $Q''$, respectively. Next, we just need to put $c\rightarrow a\rightarrow d$ in a vertical chain alone and put $b$ in a vertical chain alone. Then $Q$ can be viewed as a $\mathcal{Q}^N$ quiver.

{\bf Case 4.} If $Q$ is of Type IV of Lemma \ref{Lemma: Four types of quiver in type D}. Since $Q'$, $Q''$, $Q'''$, $\dots \in\mu^{A}$, by Proposition \ref{prop:quiver in mu A}, we can rearrange the vertices of $Q'$, $Q''$, $Q'''$, $\dots$ such that $Q'$, $Q''$, $Q'''$, $\dots$ are $\mathcal{Q}^N$ quivers, and $c'$, $c''$, $c'''$, $\dots$ lie in vertical chains of length $1$ of $Q'$, $Q''$, $Q'''$, $\dots$, respectively. Next, we only need to select a vertex from the central cycle and place it in a vertical chain alone, while putting the other vertices from the central cycle into another vertical chain. Then $Q$ can be viewed as a $\mathcal{Q}^N$ quiver.
\end{proof}

To illustrate the above lemma, we give an example as follows.

\begin{example}
Some quivers in $\mu^D$ are displayed in Figure \ref{some quivers mutation equivalent to type D}.
\begin{figure}
\centering
 \scalebox{0.99}{
\begin{minipage}[t]{0.24\textwidth}
 \scalebox{0.5}{\xymatrix{
 \\
&& 1 \ar[rdd]& \\
&& & & & 4 \ar[lld]\\
&& & 2 \ar[ldd]\ar[rrd] & & \\
&& & & &  5 \ar[uu]\\
& & 3 &
 }}
 \vspace{0.75cm}
\caption*{(a)}
\end{minipage}
\hskip0.6cm
\begin{minipage}[t]{0.24\textwidth}
 \scalebox{0.5}{\xymatrix{
&1\ar[rrd] & & & \\
&& & 6\ar[lld]\ar[ldd] \ar[rdd]&  \\
&2\ar[uu] & & & \\
&& 5\ar[rdd]& &8\ar[ldd]  \\
&3 \ar[rrd]& & & \\
&& & 7\ar[uuuu]\ar[lld]& \\
&4\ar[uu] & & &
}}
\caption*{(b)}
\end{minipage}
\begin{minipage}[t]{0.24\textwidth}
 \scalebox{0.5}{\xymatrix{
 \\
& & 4 \ar[ldd]&  \\
1 \ar[rd]& &   & 7 \ar[ld] \\
 &3 \ar[ld]\ar[rdd]&   5\ar[uu] \ar[rd] &   \\
 2 \ar[uu]& &   & 8 \ar[uu] \\
 & &  6\ar[uu] &
}}
\vspace{0.726cm}
\caption*{(c)}
\end{minipage}
\begin{minipage}[t]{0.24\textwidth}
\scalebox{0.5}{\xymatrix{
& & 4 \ar[rdddd]&   \\
& & &  \\
1\ar[rd]& & 5 \ar[uu] \ar[ld] &    \\
& 3\ar[rd] \ar[ld]&  &  \\
2\ar[uu]& &  6 \ar[uu]  &  8 \ar[ldd]  \\
&  & &  \\
& &  7 \ar[uu] &
 }}
 \vspace{0.27cm}
 \caption*{(d)}
\end{minipage}}
\caption{(a), (b), (c), (d) are quivers of Type I, Type II, Type III, Type IV in $\mu^D$, respectively.}
\label{some quivers mutation equivalent to type D}
\end{figure}
\end{example}

\begin{remark}\label{remark: the quiver in mu A}
In \cite{GM17}, Garver and Musiker gave a combinatorial approach to construct maximal green sequences for quivers in $\mu^A$. In \cite{CDRSW16}, Cormier, Dillery, Resh, Serhiyenko, and Whelan explicitly constructed minimal length maximal green sequences for quivers in $\mu^A$ and showed that the minimal length is equal to the sum of the number of vertices and the number of 3-cycles in the quiver. 
We have various methods to view a quiver $Q$ in $\mu^{A}$ as a $\mathcal{Q}^N$ quiver, different approaches lead to distinct maximal green sequences. In fact, according to Theorem \ref{Main Theorem: maximal green sequences}, the construction method illustrated in Proposition \ref{prop:quiver in mu A} ensures that the corresponding maximal green sequence of $Q$ is of minimal length.  Moreover, the maximal green sequences constructed for quivers in $\mu^A$ in this paper are different from those constructed in \cite{CDRSW16, GM17}. For instance, by Theorem \ref{Main Theorem: maximal green sequences}, the mutation sequence
\[
\mu_{3} \circ \mu_{6} \circ \mu_{5} \circ \mu_{3} \circ \mu_{2} \circ \mu_{4} \circ \mu_{7} \circ \mu_{5} \circ \mu_{3} \circ \mu_{1} \circ \mu_{2}
\]
is a maximal green sequence of the quiver in Figure \ref{example: A quiver in mu A}. On the other hand, using the approach described by Garver and Musiker, the mutation sequences
\[
\mu_{3} \circ \mu_{6} \circ \mu_{2} \circ \mu_{1} \circ \mu_{6} \circ \mu_{5} \circ \mu_{4} \circ \mu_{3} \circ \mu_{6} \circ \mu_{5} \circ \mu_{7} \circ \mu_{6}
\]
and
\[
\mu_{1} \circ \mu_{3} \circ \mu_{5} \circ \mu_{7} \circ \mu_{6} \circ \mu_{1} \circ \mu_{3} \circ \mu_{5} \circ \mu_{4} \circ \mu_{1} \circ \mu_{3} \circ \mu_{2} \circ \mu_{1}
\]
are maximal green sequences of the quiver in Figure \ref{example: A quiver in mu A}, see \cite[Remark 6.6]{GM17}.
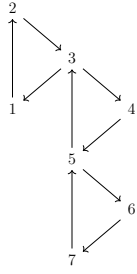
\begin{figure}
\adjustbox{scale=0.5}{\begin{tikzcd}
	2 \\
	& 3 \\
	1 && 4 \\
	& 5 \\
	&& 6 \\
	& 7
	\arrow[from=3-1, to=1-1]
	\arrow[from=1-1, to=2-2]
	\arrow[from=2-2, to=3-1]
	\arrow[from=2-2, to=3-3]
	\arrow[from=3-3, to=4-2]
	\arrow[from=4-2, to=2-2]
	\arrow[from=4-2, to=5-3]
	\arrow[from=5-3, to=6-2]
	\arrow[from=6-2, to=4-2]
\end{tikzcd}}
\caption{A quiver in $\mu^A$.}\label{example: A quiver in mu A}
\end{figure}
\end{remark}

\begin{remark}
In \cite{GMS18}, the authors provided explicit constructions of minimal length maximal green sequences for quivers in $\mu^{D}$ with the help of triangulations of an annulus or a punctured disk. In this paper, Theorem \ref{Main Theorem: maximal green sequences} provides explicit maximal green sequences for quivers in $\mu^D$. In particular, our construction method of the maximal green sequences for quivers in $\mu^D$ is different from the construction method of \cite{GMS18}. 
\end{remark}

\section*{Acknowledgements}
The authors are thankful to Changjian Fu and Jian-Rong Li for helpful discussions. The work was partially supported by the National Natural Science Foundation of China (No. 12171213, 12001254) and by Gansu Province Science Foundation for Youths (No. 22JR5RA534).


\begin{thebibliography}{99999}

\bibitem{A09} C. Amiot, Cluster categories for algebras of global dimension 2 and quivers with potential, Ann. Inst. Fourier 59 (6) (2009) 2525--2590.
				
\bibitem{BDP14}  T. Br\"ustle, G. Dupont, M. P\'erotin, On maximal green sequences, Int. Math. Res. Not. 2014 (16) (2014) 4547--4586.

\bibitem{BHIT17} T. Br\"ustle, S. Hermes, K. Igusa, G. Todorov, Semi-invariant pictures and two conjectures on maximal green sequences, J. Algebra 473 (2017) 80--109.

\bibitem{BV08} A. B. Buan, D.F. Vatne, Derived equivalence classification for cluster-tilted algebras of type $A_n$, J. Algebra 319 (7) (2008) 2723--2738.

\bibitem{B16} E. Bucher, Maximal green sequences for cluster algebras associated to orientable surfaces with empty boundary, Arnold Math. J. 2 (4) (2016) 487--510.

\bibitem{BM18} E. Bucher, M. R. Mills, Maximal green sequences for cluster algebras associated with the $n$-torus with arbitrary punctures, J. Algebraic Combin. 47 (3) (2018) 345--356.


\bibitem{BMRYZ20} E. Bucher, J.  Machacek, E. Runburg, A. Yeck,  E. Zewde, Building maximal green sequences via component preserving mutations, Ars Math. Contemp. 19 (2) (2020) 249--275.
	
\bibitem{CL19} P. Cao, F. Li, Uniform column sign-coherence and the existence of maximal green sequences, J. Algebraic Combin. 50 (4) (2019) 403--417.
			
\bibitem{CDRSW16} E. Cormier, P. Dillery, J. Resh, K. Serhiyenko, J. Whelan, Minimal length maximal green sequences and triangulations of polygons, J. Algebraic Combin. 44 (4) (2016) 905--930.

\bibitem{DWZ10} H. Derksen, J. Weyman, A. Zelevinsky, Quivers with potentials and their representations II: applications to cluster algebras, J. Amer. Math. Soc. 23 (3) (2010) 749--790.
								
\bibitem{FZ02} S. Fomin, A. Zelevinsky, Cluster algebras. I. Foundations, J. Amer. Math. Soc. 15 (2) (2002) 497--529.

\bibitem{FZ07} S. Fomin, A. Zelevinsky, Cluster algebras. IV. Coefficients, Compos. Math. 143 (1) (2007) 112--164.

\bibitem{FG23} C. Fu, S. Geng, On maximal green sequence for quivers arising from weighted projective lines, Algebr. Represent. Theory 26 (5) (2023) 1713--1729.

\bibitem{GM17} A. Garver, G. Musiker, On maximal green sequences for type $\mathbb{A}$ quivers, J. Algebraic Combin. 45 (2) (2017) 553--599.

\bibitem{GMS18} A. Garver, T. McConville, K. Serhiyenko, Minimal length maximal green sequences, Adv. in Appl. Math. 96 (2018) 76--138.

\bibitem{GHKK18} M. Gross, P. Hacking, S. Keel, M. Kontsevich, Canonical bases for cluster algebras, J. Amer. Math. Soc. 31 (2) (2018) 497--608.	
				
\bibitem{HL16} D. Hernandez, B. Leclerc, A cluster algebra approach to $q$-characters of Kirillov-Reshetikhin modules, J. Eur. Math. Soc. 18 (5) (2016) 1113--1159.

\bibitem{K1990} V. Kac, Infinite-dimensional Lie algebras, third edition, Cambridge University Press, Cambridge, 1990.

\bibitem{K11} B. Keller, On cluster theory and quantum dilogarithm identities, in: Representations of algebras and related topics, European Mathematical Society, Z\"urich, 2011, pp. 85--116.

\bibitem{K20} B. Keller,  L. Demonet, A survey on maximal green sequences, in: Representation Theory and Beyond, Contemp. Math., vol. 758, Amer. Math. Soc., Providence RI, 2020, pp. 267--286.

\bibitem{M16} G. Muller, The existence of a maximal green sequence is not invariant under quiver mutation, Electron. J. Comb. 23 (2) (2016) 2.47.

\bibitem{M17} M. R. Mills, Maximal green sequences for quivers of finite mutation type, Adv. Math. 319 (2017) 182--210.

\bibitem{M18} M. R. Mills, On Maximal Green Sequences, Local-acyclicity, and Upper Cluster Algebras, Thesis (Ph.D.)--The University of Nebraska - Lincoln, ProQuest LLC, Ann Arbor, MI, 2018.

\bibitem{S14} A. I. Seven, Maximal green sequences of exceptional finite mutation type quivers, SIGMA Symmetry Integrability Geom. Methods Appl. 10 (2014) 089.

\bibitem{Q22} F. Qin, Bases for upper cluster algebras and tropical points, J. Eur. Math. Soc. (JEMS) 26 (4) (2024) 1255--1312.

\bibitem{V10} D. F. Vatne, The mutation class of $D_n$ quivers, Comm. Algebra 38 (3) (2010) 1137--1146.
				
\end{thebibliography}
\end{document}